\documentclass[10pt,a4]{amsart}
\usepackage{amsmath}
\usepackage{amssymb}
\usepackage{amstext}
\usepackage{amscd}
\usepackage{amsthm}
\usepackage{tikz-cd}

\usepackage{color}
\usepackage{hyperref}

\usepackage[utf8]{inputenc}
\usepackage[T1]{fontenc}

\newtheorem{teo}{Theorem}[section]
\newtheorem{prop}[teo]{Proposition}
\newtheorem{lem}[teo]{Lemma}

\newtheorem{conj}[teo]{Conjecture}

\newtheorem{defini}[teo]{Definition}

\newtheorem{rem}[teo]{Remark}

\newcommand{\GL}{{\mathbf{GL}}}

\newcommand{\Sh}{{\rm Sh}}

\newcommand{\Gal}{{\rm Gal}}
\newcommand{\Res}{{\rm Res}}

\newcommand{\der}{{\rm der}}

\newcommand{\ad}{{\rm ad}}

\newcommand{\FF}{{\mathbb F}}
\newcommand{\CC}{{\mathbb C}}
\newcommand{\RR}{{\mathbb R}}
\newcommand{\ZZ}{{\mathbb Z}}
\newcommand{\QQ}{{\mathbb Q}}
\newcommand{\NN}{{\mathbb N}}

\newcommand{\GG}{{\mathbb G}}

\newcommand{\AAA}{{\mathbb A}}

\newcommand{\un}{{\rm un}}

\newcommand{\wt}{\widetilde}

\usepackage{color}

\title[Degrees of maximal special subvarieties]{Effective estimates for the degrees of maximal special subvarieties}

\author{Christopher Daw}
\address{Daw: Department of Mathematics and Statistics, University of Reading,
    White\-knights,  PO Box 217,  Reading,  Berkshire RG6 6AH,  United Kingdom}
\email{chris.daw@reading.ac.uk}

\author{Ariyan Javanpeykar}
\address{Javanpeykar: Institut f\"ur Mathematik, Johannes Gutenberg-Universit\"at, Staudingerweg 9, 55128 Mainz, Germany}
\email{peykar@uni-mainz.de}

\author{Lars K\"uhne}
\address{K\"uhne: Departement Mathematik und Informatik, Spiegelgasse 1, 4051 Basel, Switzerland}
\email{lars.kuehne@unibas.ch}

\subjclass[2010]{14G35, 11G18, 14Q20}

\keywords{Shimura varieties, Andr\'e-Oort conjecture, effectivity, degrees}

\begin{document}
\maketitle

\begin{abstract}
Let $Z$ be an algebraic subvariety of a Shimura variety. We extend results of the first author to prove an effective upper bound for the degree of a non-facteur maximal special subvariety of $Z$. 
\end{abstract}

\section{Introduction}

The motivation for this paper has been the following two conjectures. 

\begin{conj}[Andr\'e-Oort Conjecture]\label{AO}
Let $Z$ be a subvariety of a Shimura variety $S$. Then $Z$ contains only finitely many maximal special subvarieties.
\end{conj}

By a maximal special subvariety of $Z$, we refer to a special subvariety of $S$ contained in $Z$ that is not properly contained in another special subvariety of $S$ also contained in $Z$.

We denote by $\mathcal{A}_g$ the moduli space of principally polarized abelian varieties of dimension $g$, which is a Shimura variety. We denote by $\mathcal{T}_g\subseteq\mathcal{A}_g$ the Torelli locus, that is, the Zariski closure of the image $\mathcal{T}^\circ_g$ of the Torelli morphism from the moduli space $\mathcal{M}_g$ of non-singular projective curves of genus $g$ to $\mathcal{A}_g$, which sends (the isomorphism class of) a curve to (the isomorphism class of) its Jacobian.

\begin{conj}[Cf. \cite{oort:AO}, \S5]\label{oort}
For $g$ sufficiently large, $\mathcal{T}_g$ contains no positive-dimensional special subvarieties that intersect $\mathcal{T}^\circ_g$.
\end{conj}

Note that, for $S=\mathcal{A}_g$, Conjecture \ref{AO} is also raised in \cite{oort:AO}. A major motivation for both conjectures was the following question of Coleman. Let $g \geq 4$ be an arbitrary integer.

\vspace{0.2cm}
\textbf{Coleman's problem.}
\textit{Are there only finitely many isomorphism classes of non-singular projective curves of genus $g$ having CM Jacobian?}
\vspace{0.2cm}

The isomorphism classes above correspond bijectively to special points in $\mathcal{T}^\circ_g$. Therefore, Conjecture \ref{AO} in the case $S=\mathcal{A}_g$, which is now a theorem due to Tsimerman \cite{tsimerman:AO}, reduces Coleman's problem to the following question. 

\vspace{0.2cm}
\textit{Does $\mathcal{T}_g$ contain positive-dimensional special subvarieties intersecting $\mathcal{T}^\circ_g$?}
\vspace{0.2cm}

To be precise, a \textit{negative} answer to this question would imply a \textit{positive} answer to Coleman's problem. However, it has been answered positively for each low genus $g \in \{4,5,6,7\}$. That is, Coleman's problem is false for $g \in \{4,5,6,7\}$.
Nevertheless, other results have indicated that Conjecture \ref{oort} may still hold (see \cite{zuo1, zuo2}, for example). For an excellent survey of the topic, we refer the reader to \cite{Moonen2013}.

The Andr\'e-Oort conjecture was first proved under the generalised Riemann hypothesis (GRH) by Klingler, Ullmo, and Yafaev \cite{uy:andre-oort},\cite{KY:AO}. It has since been proved unconditionally for Shimura varieties of abelian type (and, in particular, $\mathcal{A}_g$) by Pila and Tsimerman \cite{pt:axlindemann}, \cite{tsimerman:AO}, using the so-called Pila-Zannier method \cite{Pila2008}. These results are, however, not effective, and hence give no quantitative control on maximal special subvarieties. In fact, effective results of Andr\'e-Oort type are rather sparse \cite{Allombert2015, Bilu2016, Bilu2017, Bilu2018, BMZ:AO, kuhne:ao, Wuestholz2014}. In this article, we complement these results by giving effective upper bounds on the degrees of certain maximal special subvarieties. More precisely, we treat the so-called non-facteur special subvarieties, which are precisely those for which we expect to find only finitely many of bounded degree (see Remark \ref{NFrem}). Establishing the latter, in an effective fashion, is a theme of current work. We note that, in a recent preprint, Binyamini \cite{Binyamini2018} has obtained similar results in a product of modular curves using differential algebraic geometry. 

Throughout this article, we take degrees with respect to the Baily-Borel line bundle (see Section \ref{BBLB} for details).

\begin{teo}\label{maintheorem}
Let $S$ be a Shimura variety satisfying the assumptions described in Section \ref{statement}, and let $Z$ be an algebraic subvariety of $S$. Then there exists an effectively computable constant $c=c(\dim(Z),\deg(Z))$ such that, if $V$ is a non-facteur maximal special subvariety of $Z$, then the degree of $V$ is at most $c$.
\end{teo}

The assumptions referred to are modest: we assume that the group $\mathbf G$ defining $S$ is semisimple of adjoint type and that the associated compact open subgroup $K$ of $\mathbf G(\AAA_f)$ is equal to a product of compact open subgroups $K_p\subset \mathbf G(\QQ_p)$. 

To obtain Theorem \ref{maintheorem}, we first generalize earlier work of the first author \cite{daw:degrees}, in which lower bounds for the degrees of so-called strongly special subvarieties were given. In this paper, we obtain a lower bound for the degree of any positive-dimensional special subvariety, in terms of a group-theoretic product of primes (see Theorem \ref{mainbound}). The main tool is again Prasad's volume formula for $S$-arithmetic quotients of semisimple groups \cite{prasad:volumes}.

The article \cite{daw:degrees} gave a new proof of a finiteness theorem originally obtained by Clozel and Ullmo concerning strongly special subvarieties (see \cite{CU}, Theorem 1.1). This theorem was subsequently generalized by Ullmo to non-facteur special subvareties (see \cite{ullmo:equidistribution}, Theorem 1.3), and the main result of this paper is a significant step towards an effective proof of that result (as mentioned above, it remains to give an effective procedure for obtaining the (conjecturally) finitely many non-facteur special subvarieties of bounded degree.

The purpose of \cite{daw:degrees} was to give a new proof of the Andr\'e-Oort conjecture under GRH, combining results contained therein with the arithmetic side of the Klingler-Ullmo-Yafaev approach. Indeed, this was achieved via the technique involving Hecke correspondences that was initially conceived by Edixhoven and substantially generalized by Klingler and Yafaev. In this article, we generalize the results obtained in \cite{daw:degrees} (see Theorems \ref{hecke} and \ref{criterion}), but, rather than applying them in the previous manner, in which one takes a set of special subvarieties and incrementally increases the dimension of its members, we proceed slightly differently, using Hecke correspondences to perform a cutting-out procedure.

Non-facteur special subvarieties were defined in \cite{ullmo:equidistribution}. They are those special subvarieties that do not arise in infinite families: a special subvariety $V$ of a Shimura variety $S$ is called non-facteur if there exists no finite morphism of Shimura varieties $S_1\times S_2\rightarrow S$, with $S_2$ having positive dimension, such that $V$ is equal to the image of $S_1\times\{z\}$ in $S$ for any (necessarily special) point $z\in S_2$. The reader might consider the difference between the image of the modular curve $Y_0(N)$ in $\CC^2$ and the fibre $\CC\times\{z\}$ for a CM point $z\in\CC$; the former is non-facteur, the latter is not. 

We emphasise that, in the case of $S=\mathcal{A}_g$, there exist effectively computable degree bounds for $\mathcal{T}_g$ (see, for example, \cite[(6.8)]{Grushevsky2004}), so Theorem \ref{maintheorem} can indeed be used to produce explicit bounds for the degrees of non-facteur maximal special subvarieties of $\mathcal{T}_g$.

\section*{Acknowledgements}
The first author would like to thank the EPSRC and the Institut des Hautes Etudes Scientifiques for granting him a William Hodge fellowship, during which he first began working on this topic. He would like to thank the EPSRC again, as well as Jonathan Pila, for the opportunity to be part of the project Model Theory, Functional Transcendence, and Diophantine Geometry. He would like the University of Reading for its ongoing support. He would like to both the second and third authors, and their institutions, for invitations to visit. 

The second author gratefully acknowledges support from SFB/Transregio 45. He would also like to thank Manfred Lehn and Kang Zuo for helpful discussions.

The third author acknowledges support from the Swiss National Science Foundation through an Ambizione grant (no. 168055). He also would like to thank Philipp Habegger for discussions and encouragement.

The authors would like to collectively thank Stefan M\"uller-Stach for having originally suggested that they work on this problem, and for many helpful discussions. They also thank the referee for several insightful comments and suggestions.

\section{Preliminaries}

First we establish some general conventions.

\subsection{Fields} For a number field $F$ and a place $v$ of $F$, we let $F_v$ denote the completion of $F$ with respect to $v$. In particular, for a rational prime $p$, we let $\QQ_p$ denote the $p$-adic numbers. We denote by $\AAA_f$ the finite rational ad\`eles.

\subsection{Groups} For an algebraic group ${\mathbf G}$, we denote by ${\mathbf G}^\circ$ the connected component of ${\mathbf G}$ containing the identity. We denote by ${\mathbf G}^\der$ the derived subgroup of ${\mathbf G}$ and by ${\mathbf G}^\ad$ the quotient of ${\mathbf G}$ by its center. By the rank of ${\mathbf G}$, we refer to the dimension of a maximal torus of ${\mathbf G}$.

For an algebraic subgroup ${\mathbf H}$ of ${\mathbf G}$, we denote by $\mathbf{Z(H)}$ the center of ${\mathbf H}$ and by ${\mathbf Z}_{\mathbf G}({\mathbf H})$ the centralizer of $\mathbf H$ in $\mathbf G$. When $\mathbf H$ is defined over $\QQ$ or $\RR$, we let $\mathbf H(\RR)^+$ denote the connected component of $\mathbf H(\RR)$ containing the identity, and we let $\mathbf H(\RR)_+$ denote the inverse image of $\mathbf H^{\ad}(\RR)^+$ in $\mathbf H(\RR)$. We let $\mathbf H(\QQ)^+$ denote $\mathbf H(\QQ)\cap \mathbf H(\RR)^+$ and we let $\mathbf H(\QQ)_+$ denote $\mathbf H(\QQ)\cap \mathbf H(\RR)_+$. If $\mathbf H$ is finite, we denote the (finite) cardinality of $\mathbf H(\CC)$ by $|\mathbf H|$.

\subsection{Representations}

Let $\mathbf G$ be a reductive algebraic group over a $\QQ$ and fix a faithful representation $\rho: \mathbf G\rightarrow\GL_n$. Given these data, we will often identify $\mathbf G$ and its algebraic subgroups with their images in $\GL_n$. For any $\QQ$-subgroup $\mathbf H$ of $\mathbf G$ and any prime $p$, we let $\mathbf{H}_{\ZZ_p}$ denote the Zariski closure of $\mathbf H$ in $\GL_{n,\ZZ_p}$. For a rational prime $p$ and a subtorus $\mathbf T$ of $\mathbf G_{\QQ_p}$, we denote by $X^*(\mathbf T)$ the character module of $\mathbf T$ (that is, the free $\ZZ$-module of finite rank comprising the homomorphisms $\mathbf T_{\overline{\QQ}_p}\rightarrow\GG_{m,\overline{\QQ}_p}$, where $\overline{\QQ}_p$ denotes an algebraic closure of $\QQ_p$). We say that a character $\chi\in X^*(\mathbf T)$ intervenes in $\rho$ if there exists a non-zero subspace $V_\chi\subseteq\overline{\QQ}^n_p$ such that $\rho(t)v=\chi(t)v$ for any $v\in V_\chi$ and any $t\in \mathbf T(\overline{\QQ}_p)$. 

\subsection{Neat subgroups} 

We briefly recall the definition of neat subgroups. References are Chapter 17 in \cite{borel:arithmetic-groups} and Section 4.1.4 in \cite{KY:AO}. Let $\mathbf G$ be a linear algebraic subgroup defined over $\QQ$ and $\rho: \mathbf G \rightarrow \GL_n$ a faithful representation. For an element $g$ of $\mathbf G(\QQ)$ (resp.\ $\mathbf G(\QQ_p)$), we write  $\mathrm{Eig}(\rho(g))$ for the subgroup generated in $\overline{\QQ}$ (resp.\ $\overline{\QQ}_p$) by the eigenvalues of $\rho(g)$. For $g \in \mathbf G(\QQ_p)$, we furthermore set
\begin{equation*}
\mathrm{Eig}_{\mathrm{alg}}(\rho(g)) = \{ \alpha \in \mathrm{Eig}(\rho(g)) \ | \ \text{$\alpha$ is algebraic over $\QQ$} \}.
\end{equation*}
We then say that an element $g \in \mathbf G(\QQ)$ (resp.\ $g \in \mathbf G(\QQ_p)$) is neat if $\mathrm{Eig}(\rho(g))$ (resp.\ $\mathrm{Eig}_{\mathrm{alg}}(\rho(g))$) is torsion-free. An element $g = (g_p) \in \mathbf G(\mathbb{A}_f)$ is called neat if for each integer $N>1$ there exists a rational prime $p$ such that $\mathrm{Eig}_{\mathrm{alg}}(\rho(g)_p)$ contains no element of (exact) order $N$. These notations are independent of the choice of $\rho$. 

A subgroup $\Gamma\subset \mathbf G(\QQ)$ or $K \subset \mathbf G(\mathbb{A}_f)$ is said to be \emph{neat} if every element of $\Gamma$ or $K$, respectively, is neat.

\subsection{Products of primes}
\label{subsection_primes}

Let $\mathbf G$ be a reductive algebraic group over $\QQ$ and let $K\subset \mathbf G(\AAA_f)$ be a compact open subgroup equal to the product of compact open subgroups $K_p\subset \mathbf G(\QQ_p)$. We denote by $\Sigma(\mathbf G)$ the set of all primes $p$ such that, either 
\begin{itemize}
\item $\mathbf G_{\QQ_p}$ is not quasi-split (equivalently, $\mathbf{G}_{\QQ_p}$ does not contain a $\QQ_p$-Borel subgroup), or
\item $\mathbf G_{\QQ_p}$ is quasi-split, but does not split over an unramified extension of $\QQ_p$.
\end{itemize}
We denote by $\Sigma(\mathbf G,K)$ the union of $\Sigma(\mathbf G)$ and the set of primes such that
\begin{itemize}
\item $\mathbf G_{\QQ_p}$ is quasi-split, splits over an unramified extension of $\QQ_p$, but $K_p$ is not hyperspecial (equivalently, there is no smooth reductive $\ZZ_p$-group scheme $\mathcal G$ such that ${\mathcal{G}}_{\QQ_p}= \mathbf G_{\QQ_p}$ and ${\mathcal G}(\ZZ_p)=K_p$). 
\end{itemize}

As explained in \cite{daw:degrees}, Section 6, we have $K_p=\mathbf G_{\ZZ_p}(\ZZ_p)$ for almost all $p$, from which it follows that $\Sigma(\mathbf{G},K)$ is a finite set of primes. We will denote by $\Pi(\mathbf{G},K)$ the product of the primes contained in $\Sigma(\mathbf{G},K)$, which we define to be $1$ if $\Sigma(\mathbf{G},K)$ happens to be empty. Sometimes we will insist that $K$ be neat (as defined above). 


\subsection{Degrees}

Let $S$ be a projective variety and let $Z\subseteq S$ be an irreducible subvariety. Henceforth, by a  subvariety, we will refer to an irreducible subvariety unless stated otherwise (for example, if preceded by the word Shimura). For any line bundle $L$ on $S$, we will define the degree $\deg_LZ$ of $Z$ with respect to $L$ as usual (as in \cite{Fulton1998}, for example; see also \cite{KY:AO}, Section 5.1). We will make frequent use of the projection formula (see \cite{Fulton1998}, Proposition 2.5 (c); see also \cite{KY:AO}, Section 5.1).

\section{Shimura varieties} 

We assume that the reader is somewhat familiar with the theory of Shimura varieties, and we recall only the definitions and facts essential for our arguments. We refer to \cite{milne:intro} for more details.

Let $(\mathbf{G},X)$ denote a Shimura datum and let $K$ be a compact open subgroup of $\mathbf{G}(\AAA_f)$. The double quotient space
\begin{align}\label{doublecoset}
\mathbf{G}(\QQ)\backslash X\times (\mathbf{G}(\AAA_f)/K)
\end{align}
is the analytification of a complex quasi-projective algebraic variety that we will denote $\Sh_K(\mathbf{G},X)$. (Recall that $\Sh_K(\mathbf{G},X)$ possesses a model defined over a number field $E:=E(\mathbf{G},X)$, but we will not need this fact.) Note that the action of $\mathbf{G}(\QQ)$ on the product in (\ref{doublecoset}) is the diagonal one. A variety of the form $\Sh_K(\mathbf{G},X)$ is called a Shimura variety. 

In general, the Shimura variety $\Sh_K(\mathbf{G},X)$ is not connected. However, if we choose a connected component $X^+$ of $X$, then the image $S_K(\mathbf{G},X^+)$ of $X^+\times\{1\}$ in $\Sh_K(\mathbf{G},X)$ is a connected component of $\Sh_K(\mathbf{G},X)$. Note that $S_K(\mathbf{G},X^+)$ is canonically isomorphic to $\Gamma\backslash X^+$, where $\Gamma:=\mathbf{G}(\QQ)_+\cap K$.

\subsection{Hecke correspondences}

Let $K'\subset \mathbf{G}(\AAA_f)$ be a compact open subgroup contained in $K$. The natural map on double coset spaces comes from a finite morphism
\begin{align*}
\pi_{K',K}:\Sh_{K'}(\mathbf{G},X)\rightarrow\Sh_K(\mathbf{G},X).
\end{align*}
If $\alpha\in \mathbf{G}(\AAA_f)$, then the map on double coset spaces induced by $(x,g)\mapsto(x,g\alpha)$ (where $x\in X$ and $g\in \mathbf{G}(\AAA_f)$) also comes from an isomorphism
\begin{align*}
[\alpha]:\Sh_K(\mathbf{G},X)\rightarrow\Sh_{\alpha^{-1}K\alpha}(\mathbf{G},X).
\end{align*}
We let $T_\alpha$ denote the map on algebraic cycles of $\Sh_K(\mathbf{G},X)$ coming from the algebraic correspondence
\begin{align*}
\Sh_K(\mathbf{G},X)\leftarrow\Sh_{K\cap \alpha K\alpha^{-1}}(\mathbf{G},X)\rightarrow \Sh_{\alpha^{-1}K\alpha\cap K}(\mathbf{G},X)\rightarrow\Sh_K(\mathbf{G},X),
\end{align*}
where the left and right outer arrows are $\pi_{K\cap \alpha K\alpha^{-1},K}$ and $\pi_{\alpha^{-1}K\alpha\cap K,K}$, respectively, and the middle arrow is $[\alpha]$. We refer to the map $T_\alpha$ as a Hecke correspondence.

\subsection{The Baily-Borel line bundle}\label{BBLB}

By \cite{bb:compactification}, Lemma 10.8, the line bundle of holomorphic forms of maximal degree on $X$ descends to $\Sh_K(\mathbf{G},X)$ and extends uniquely to an ample line bundle $L_K$ on the Baily-Borel compactification $\overline{\Sh_K(\mathbf{G},X)}$ of $\Sh_K(\mathbf{G},X)$. We refer to $L_K$ as the Baily-Borel line bundle on $\overline{\Sh_K(\mathbf{G},X)}$. Given a subvariety $Z$ of $\Sh_K(\mathbf{G},X)$ and its Zariski closure $\overline{Z}$ in $\overline{\Sh_K(\mathbf{G},X)}$, we will write $\deg_{L_K}Z$ instead of $\deg_{L_K}\overline{Z}$. In fact, we will simply write $\deg Z$ if it is clear to which line bundle we are referring.

As stated in \cite{KY:AO}, Proposition 5.3.2 (1), if $K'\subset \mathbf{G}(\AAA_f)$ is a compact open subgroup contained in $K$, then the pullback $\pi_{K',K}^*L_K$ of $L_K$ is equal to $L_{K'}$. 

\subsection{Special subvarieties}\label{section:spsub}

For any $\QQ$-subgroup $\mathbf H$ of $\mathbf G$, we obtain a compact open subgroup $\mathbf H(\AAA_f)\cap K$ of $\mathbf{H}(\AAA_f)$, which we denote $K(\mathbf H)$. If $(\mathbf H,X_{\mathbf H})$ is a Shimura subdatum of $(\mathbf G,X)$, the natural map
\begin{align*}
\iota_{\mathbf{H}}:\Sh_{K(\mathbf H)}(\mathbf H,X_{\mathbf H})\rightarrow\Sh_K(\mathbf G,X),
\end{align*}
on double coset spaces extends to an algebraic morphism on the Baily-Borel compactifications. In particular, it is a closed immersion.
We refer to any irreducible component of the image of such a map as a Shimura subvariety of $\Sh_K(\mathbf G,X)$. If $V$ is a Shimura subvariety of $\Sh_K(\mathbf G,X)$ and $\alpha\in \mathbf G(\AAA_f)$, we refer to any irreducible component of the algebraic cycle $T_\alpha(V)$ as a special subvariety of $\Sh_K(\mathbf G,X)$. (By abuse of notation, we identify cycles with their underlying support.)

\begin{lem}\label{independence}
Let $V$ be a special subvariety of $\Sh_K(\mathbf G,X)$. Then there exists a Shimura subdatum $(\mathbf H,X_{\mathbf H})$ of $(\mathbf G,X)$, a connected component $X^+_{\mathbf H}$ of $X_{\mathbf H}$, and $\alpha\in G(\AAA_f)$ such that $V$ is equal to the image of $X^+_{\mathbf H}\times\{\alpha\}$ in $\Sh_K(\mathbf{G},X)$. Furthermore, we can choose $(\mathbf H,X_{\mathbf H})$ such that $\mathbf H$ is the generic Mumford-Tate group of $X_{\mathbf{H}}$. That is, such that $\mathbf H$ is the smallest $\QQ$-subgroup of $\mathbf G$ such that every element of $X_{\mathbf H}$ factors through ${\mathbf H}_\RR$.
\end{lem}

\begin{proof}
The lemma is a modest generalization of \cite{uy:andre-oort}, Lemma 2.1. Hence, we imitate the proof of the latter.

Let $v\in V$ be a Hodge generic point of $V$ and let $(x,\alpha)\in X\times\mathbf G$ be a point lying above $v$. Let $\mathbf H$ be the Mumford-Tate group of $x$, let $X_{\mathbf H}=\mathbf H(\RR)x$, and let $X^+_{\mathbf H}=\mathbf H(\RR)^+x$ be the connected component of $X_{\mathbf H}$ containing $x$. Then $(\mathbf H,X_{\mathbf H})$ is a Shimura subdatum of $(\mathbf G,X)$ (see \cite{ullmo:equidistribution}, Lemme 3.3) and so the image $V'$ of $X^+_{\mathbf H}\times\{\alpha\}$ in $\Sh_K(\mathbf G,X)$ is a special subvariety containing $v$. As $v$ is Hodge generic in $V$, it follows that $V$ is the smallest special subvariety of $\Sh_K(\mathbf G,X)$ containing $v$. Therefore, $V\subseteq V'$. Since $V$ and $V'$ are irreducible and both of dimension $\dim X_{\mathbf H}$, we conclude that $V=V'$.
\end{proof}

\begin{defini}
For $V$, $(\mathbf H,X_{\mathbf H})$, $X^+_{\mathbf H}$, and $\alpha\in G(\AAA_f)$ as above, we say that $V$ is defined by $(\mathbf H,X_{\mathbf H})$, $X^+_{\mathbf H}$, and $\alpha\in G(\AAA_f)$, and we say that $V$ is associated with $(\mathbf H,X_{\mathbf H})$. In particular, these notions include the condition that $\mathbf H$ is the generic Mumford-Tate group of $X_{\mathbf{H}}$. By \cite{uy:andre-oort}, Lemma 2.1,  when $V$ is contained in $S_K(\mathbf G,X^+)$, we may assume that $X^+_{\mathbf H}$ is contained in $X^+$ and $\alpha=1$.  In this situation, we will say that $V$ is a special subvariety of $S_K(\mathbf G,X^+)$, and we will say that $V$ is defined by $(\mathbf H,X_{\mathbf H})$ and $X^+_{\mathbf H}$.
\end{defini}

\begin{lem}\label{Qconj}
Suppose that $V$ is a special subvariety of $\Sh_K(\mathbf G,X)$ associated with $(\mathbf H_1,X_{\mathbf H_1})$ and also with $(\mathbf H_2,X_{\mathbf H_2})$. Then $\mathbf H_1=g\mathbf H_2g^{-1}$ for some $g\in\mathbf G(\QQ)$.
\end{lem}

\begin{proof} By definition, $V$ is simultaneously equal to the images of $X^+_{\mathbf H_1}\times\{\alpha_1\}$ and $X^+_{\mathbf H_2}\times\{\alpha_2\}$ in $\Sh_K(\mathbf G,X)$, where $X^+_{\mathbf H_1}$ and $X^+_{\mathbf H_2}$ are connected components of $X_{\mathbf H_1}$ and $X_{\mathbf H_2}$, respectively, and $\alpha_1,\alpha_2\in\mathbf G(\AAA_f)$. By assumption, there exists a point $x_1\in X^+_{\mathbf H_1}$ whose Mumford-Tate group is equal to $\mathbf H_1$. As the image of $(x_1,\alpha_1)$ in $\Sh_K(\mathbf G,X)$ is also contained in the image of $X_{\mathbf H_2}^+ \times \{ \alpha_2\}$, there exists $g\in \mathbf G(\QQ)$ and a point $x_2\in X^{+}_{\mathbf H_2}$ such that $x_1=gx_2$. It follows that the Mumford-Tate group $\mathbf H_1$ of $x_1$ is equal to the Mumford-Tate group of $gX_2$, which is contained in $g\mathbf H_2g^{-1}$. In other words, $H_1$ is contained in $gH_2g^{-1}$. Similarly, we deduce that there exists $q\in \mathbf G(\QQ)$ such that $\mathbf H_2$ is contained in $q\mathbf H_1q^{-1}$. Therefore, $\mathbf H_1$ is contained in $gq\mathbf H_1(gq)^{-1}$ and, therefore, they are equal. The result follows immediately.
\end{proof}

\begin{lem}\label{Gammaconj}
Suppose that $V$ is a special subvariety of $S_K(\mathbf G,X^+)$ defined by $(\mathbf H_1,X_{\mathbf H_1})$ and $X^+_{\mathbf H_1}$ and also by $(\mathbf H_2,X_{\mathbf H_2})$ and $X^+_{\mathbf H_2}$. Then $\mathbf H_1=\gamma\mathbf H_2\gamma^{-1}$ for some $\gamma\in\Gamma$.
\end{lem}

\begin{proof}
By definition, $V$ is simultaneously equal to the images of $X^+_{\mathbf H_1}\times\{1\}$ and $X^+_{\mathbf H_2}\times\{1\}$ in $S_K(\mathbf G,X^+)$. By assumption, there exists a point $x_1\in X^+_{\mathbf H_1}$ whose Mumford-Tate group is equal to $\mathbf H_1$. As the image of $(x_1,1)$ in $S_K(\mathbf G,X^+)$ is also contained in the image of $X_{\mathbf H_2}^+ \times \{ 1\}$, there exists $g\in \mathbf G(\QQ)\cap K$ and a point $x_2\in X^{+}_{\mathbf H_2}$ such that $x_1=gx_2$. Finally, the fact that $x_1,x_2\in X^+$ forces $g\in\mathbf G(\QQ)_+$ and, hence, $g\in\Gamma$.
\end{proof}

\begin{defini}
Suppose that V is a special subvariety of $S_K(\mathbf G,X^+)$ defined by $(\mathbf H_1,X_{\mathbf H_1})$ and $X^+_{\mathbf H_1}$. We will write
\begin{align*}
\Sigma_V=\Sigma(\mathbf H^\der,K(\mathbf H^\der))\text{ and }\Pi_V=\Pi(\mathbf H^\der,K(\mathbf H^\der)).
\end{align*} 
These are well-defined by Lemma \ref{Gammaconj}. 
\end{defini}

\subsection{Working in the derived group}\label{derived}
Suppose that $\mathbf{Z(G)}(\RR)$ is compact.
Let $V$ denote a special subvariety of $S_K(\mathbf G,X^+)$ defined by $(\mathbf H,X_{\mathbf H})$ and $X^+_{\mathbf H}$. Then $V$ is equal to the image of $\Gamma(\mathbf H)\backslash X^+_{\mathbf H}$ in $\Gamma\backslash X^+=S_K(\mathbf G,X^+)$, where here and henceforth we write $\Gamma(\mathbf H)$ for $\Gamma\cap\mathbf H(\QQ)_+$. 

If $K$ is neat, then $\Gamma(\mathbf H)$ is neat. Furthermore, it is contained in $\mathbf H^\der(\QQ)$. To see this latter claim, let $\mathbf C$ denote the maximal $\QQ$-torus quotient of $\mathbf H$. Since $\mathbf H$ is the almost direct product of $\mathbf H^\der$ with the $\QQ$-torus $\mathbf{Z(H)}^\circ$, we obtain an isogeny $\mathbf{Z(H)}^\circ \rightarrow \mathbf{C}$.
By \cite{uy:andre-oort}, Remark 2.3, $\mathbf{Z(H)}(\RR)$ is compact and, by \cite{milne:intro}, Proposition 5.1,
\begin{align*}
\mathbf{Z(H)}(\RR)^+\rightarrow \mathbf C(\RR)^+
\end{align*} 
is surjective. Therefore, since, by \cite{milne:intro}, Corollary 5.3, $\mathbf C(\RR)$ has only finitely many connected components, we conclude that $\mathbf C(\RR)$ is compact.

On the other hand, by \cite{milne:intro}, Proposition 3.2, the image of $\Gamma(\mathbf H)$ in $\mathbf C(\QQ)$ under the natural morphism is an arithmetic subgroup, which, by \cite{borel:arithmetic-groups}, Corollaire 17.3, is neat. We conclude, then, that the image is trivial and, therefore, that $\Gamma(\mathbf H)$ is contained in $\mathbf H^\der(\QQ)$, as claimed.

\subsection{Non-facteur special subvarieties}

Let $V$ be a special subvariety of $\Sh_K(\mathbf G,X)$ associated with a Shimura subdatum $(\mathbf H,X_{\mathbf H})$. Following Ullmo (see \cite{ullmo:equidistribution}), we say that $V$ is non-facteur if the image of $\mathbf Z(\mathbf G)({\mathbf H}^{\der})(\RR)$ in $\mathbf G^{\ad}(\RR)$ is compact. By Lemma \ref{independence}, this definition is independent of the Shimura subdatum defining $V$. (The definition in the introduction is equivalent, and more intuitive, but this definition turns out to be more useful for our purposes.) 

The definition of a non-facteur special subvariety generalizes the notion of a strongly special subvariety (strongly special subvarieties were defined in \cite{CU}).

\begin{lem}
If $V$ is strongly special, then $V$ is non-facteur. 
\end{lem}

\begin{proof}
Since $V$ is strongly special it is, by definition, associated with a Shimura subdatum $(\mathbf H,X_{\mathbf H})$ of $(\mathbf G,X)$ such that the image of $\mathbf H$ in $\mathbf G^\ad$ is semisimple. Therefore, since both definitions rely on passing to $\mathbf G^\ad$, we can and do assume that $\mathbf G=\mathbf G^\ad$ from the outset. Then, $\mathbf H=\mathbf H^\der$ is semisimple and any $x\in X_\mathbf H$ factors through $\mathbf H^\der_\RR=\mathbf H_\RR$. Therefore, $\mathbf{Z(G)(H^\der)}(\RR)$ stabilizes $x\in X$ and, since (when $\mathbf G=\mathbf G^\ad$) the stabilizer of any point in $X$ is compact, the result follows.
\end{proof}

Recall that $\mathbf G^{\ad}$ is equal to the product $\mathbf G_1\times\cdots\times \mathbf G_n$ of its $\QQ$-simple factors. Furthermore, for each $i=1,\ldots,n$, there exists an almost $\QQ$-simple normal subgroup $\widetilde{\mathbf G}_i\subseteq \mathbf G^\der$ such that $\mathbf G^\der$ is the almost direct product $\widetilde{\mathbf G}_1\cdots\widetilde{\mathbf G}_n$ and
\begin{align*}
\widetilde{\mathbf G}_i\hookrightarrow \mathbf G^\der\rightarrow \mathbf G^\ad\rightarrow \mathbf G_j
\end{align*}
is a central isogeny for $i=j$ and trivial if $i\neq j$.

\begin{lem}\label{nonfacnontriv}
If $V$ is non-facteur, then the image of $\mathbf H^\der\subseteq \mathbf G^\der$ under each of the natural projections $\mathbf G^\der\rightarrow \mathbf G_i$ is non-trivial.
\end{lem}

\begin{proof}
If the image of $\mathbf H^\der$ under $\mathbf G^\der\rightarrow \mathbf G_i$ were trivial, then $\widetilde{\mathbf G}_i$ would be contained in $\mathbf{Z(G)(H^\der)}$. However, by the definition of a Shimura datum, $\widetilde{\mathbf G}_i(\RR)$ is not compact. Hence, we obtain at a contradiction.
\end{proof}

\begin{rem}\label{NFrem}
We expect that there are only finitely many non-facteur special subvarieties of degree at most $A$, though we believe this to be an open problem. There are likely many ways to approach this problem, but we outline our intuition below. 

Let $V$ be a non-facteur special subvariety of $\Sh_K(\mathbf G,X)$ satisfying $\deg V\leq A$. Then $V$ is associated with a Shimura subdatum $(\mathbf H,X_\mathbf H)$. However, since the image of $\mathbf{Z(G)(H^\der)}(\RR)$ in $\mathbf G^\ad(\RR)$ is compact, it follows that the $\mathbf H(\RR)$-conjugacy class $X_\mathbf H$ is equal to the $\mathbf H_1(\RR)$-conjugacy class containing it, where $\mathbf H_1$ is the almost direct product of $\mathbf H^\der$ and $\mathbf{Z(G)(H^\der)}^\circ$. If we let $\mathbf H_2$ denote the product of $Z(\mathbf H_1)^\circ$ with the almost $\QQ$-simple factors of $\mathbf H_1$ whose underlying real Lie groups are non-compact, then \cite{ullmo:equidistribution}, Lemme 3.3, implies that $(\mathbf H_2,X_\mathbf H)$ is a Shimura subdatum of $(\mathbf G,X)$, and $V$ is associated with it. On the other hand, \cite{DR}, Conjecture 10.4, predicts that there exists a finite set $\Omega:=\Omega(A)$ of semisimple $\QQ$-subgroups of $\mathbf G$ such that $\mathbf H^\der=\gamma \mathbf F\gamma^{-1}$, for some $\gamma\in\Gamma$ and $\mathbf F\in\Omega$. Therefore, $\mathbf H_2=\gamma \mathbf F_2\gamma^{-1}$, where $\mathbf F_2$ is to $\mathbf F$ as $\mathbf H_2$ is to $\mathbf H^\der$. We conclude that $V$ is associated with a Shimura datum of the form $(\mathbf F_2,Y)$. By \cite{uy:andre-oort}, Lemma 3.7, there are only finitely many such Shimura subdata and this restricts $V$ to a finite set.
\end{rem}

\section{Comparing degrees of special subvarieties}
The following two results comparing the degrees of special subvarieties with respect to different Baily-Borel line bundles will be used repeatedly in the article. The first is immediate from \cite{KY:AO}, Corollary 5.3.10.
	
	\begin{lem}\label{upperbound} Let $(\mathbf G,X)$ be a Shimura datum such that $\mathbf{Z(G)}(\RR)$ is compact and let $K\subset \mathbf \mathbf \mathbf G(\AAA_f)$ be a neat compact open subgroup. Furthermore, let $(\mathbf H,X_{\mathbf H})$ be a Shimura subdatum of $(\mathbf G,X)$.
	For each Hodge-generic subvariety $Y_{\mathbf H}$ of $\Sh_{K(\mathbf H)}(\mathbf H,X_{\mathbf H})$ with image $Y$ in $\Sh_K(\mathbf G,X)$, we have $$\deg_{L_{K(\mathbf H)}}Y_{\mathbf H} \leq \deg_{L_K}Y.$$
	\end{lem}
	
The second provides a certain converse.
	
	\begin{lem} 
	\label{lowerbound}	
	Let $(\mathbf G,X)$ and $K$ be as in Lemma \ref{upperbound} and let $X^+$ denote a connected component of $X$. Then there exists an effectively computable constant $d$ such that the following holds:
		
	Let $(\mathbf H,X_{\mathbf H})$ be a Shimura subdatum of $(\mathbf G,X)$ and let $X^+_{\mathbf H}$ denote a connected component of $X_{\mathbf H}$ contained in $X^+$. For each subvariety $Y_{\mathbf H}$ of $S_{K(\mathbf H)}(\mathbf H,X_{\mathbf H}^+)$ with image $Y$ in $S_K(\mathbf G,X^+)$, we have $$\deg_{L_K}Y\leq d \cdot \deg_{L_{K(\mathbf H)}}Y_{\mathbf H}$$.
	\end{lem} 

	Note that there is no Hodge-generic assumption on $Y_{\mathbf{H}}$ in Lemma \ref{lowerbound}.

	\begin{proof}
		We use the term uniform to refer to any constant depending only on $(\mathbf{G}, X)$ and $K$.
		
		First note that we can and do assume that $Y_\mathbf H$ is Hodge generic in $S_{K(\mathbf H)}(\mathbf H,X_\mathbf H^+)$. To see this, let $(\mathbf{H^\prime},X_\mathbf{H^\prime})$ denote a Shimura subdatum of $(\mathbf H,X_\mathbf H)$ and let $X^+_\mathbf{H^\prime}$ denote a connected component of $X_\mathbf{H^\prime}$ such that $Y_\mathbf H$ is the image of a Hodge generic subvariety $Y_\mathbf{H^\prime}$ of $S_{K(\mathbf{H^\prime})}(\mathbf{H^\prime},X_\mathbf{H^\prime}^+)$ under the natural morphism. If we have
		\begin{align*}
		\deg_{L_K}Y\leq d\cdot\deg_{L_{K(\mathbf{H^\prime})}}Y_\mathbf{H^\prime},
		\end{align*}
		for some uniform, effectively computable constant $d$, then the assertion follows from Lemma \ref{upperbound} above.
		
		Now let
		\begin{align*}
		\iota_\mathbf H:\Sh_{K(\mathbf H)}(\mathbf H,X_\mathbf H)\rightarrow\Sh_K(\mathbf G,X)
		\end{align*}
		denote the natural morphism of Shimura varieties. By the proof of \cite{uy:andre-oort}, Lemma 2.2, the morphism $\iota_{\mathbf H|{Y_\mathbf H}}$ is generically injective and so, by the projection formula, we have 
		\begin{align*}
		\deg_{L_K}Y=\deg_{\iota_\mathbf H^*{L_K}}Y_\mathbf H.
		\end{align*}
		Hence, it remains to show that 
		\begin{align}\label{degreescomp}
		\deg_{\iota_\mathbf H^*{L_K}}Y_\mathbf H\leq d\cdot\deg_{L_{K(\mathbf H)}}Y_\mathbf H,
		\end{align}
		for some uniform, effectively computable constant $d$.
		
		Note that the connected component $S_{K(\mathbf H)}(\mathbf H,X_\mathbf H^+)$ is equal to $\Gamma(\mathbf H)\backslash X^+_\mathbf H$, where, by Section \ref{derived}, $\Gamma(\mathbf H)$ is an arithmetic subgroup of the semisimple real Lie group $\mathbf H^\der(\RR)$. By \cite{satake:embeddings}, Section 4, the Baily-Borel line bundle $L_{K(\mathbf H)}$ on $S_{K(\mathbf H)}(\mathbf H,X_\mathbf H^+)$ is defined by a tuple ${\bf m}_\mathbf H$ of non-negative integers. In fact, by \cite{satake:embeddings}, Section 4.3, since the corresponding automorphy factor is a positive integer power of the functional determinant (see \cite{bb:compactification}, Section 7.3), ${\bf m}_\mathbf H$ has positive entries. Similarly, $S_K(\mathbf G,X^+)$ is equal to $\Gamma\backslash X^+$, where $\Gamma$ is an arithmetic subgroup of the semisimple real Lie group $\mathbf G(\RR)$, and the Baily-Borel line bundle $L_K$ on $S_K(\mathbf G,X^+)$ is defined by another tuple $\bf m$ of positive integers. As explained in \cite{satake:embeddings}, Section 4.1, the automorphic line bundle $L:=\iota_\mathbf H^*L_K$ is defined by ${\bf m}M$ for some matrix $M$ associated with the inclusion of $\mathbf H^\der$ in $\mathbf G$. However, as can be seen from \cite{satake:embeddings}, Sections 1.4 and 2.1, the entries of $M$ are all integers belonging to the set $\{0,\ldots,t\}$, where $t$ is any bound for the maximum number of simple $\RR$-roots of any $\RR$-simple factor of $\mathbf G_\RR$. In particular, the entries of ${\bf m}M$ are at most $d:=r^3_\mathbf G$, where $r_\mathbf G$ denotes the rank of $\mathbf G$. We conclude that the entries of $d{\bf m}_\mathbf H$ are all greater than the corresponding entries of ${\bf m}M$.
		
		The automorphic line bundle $\Lambda:=L_{K(\mathbf H)}^{\otimes d}\otimes L^{-1}$ is associated with the tuple $d{\bf m}_\mathbf H-{\bf m}M$. This tuple is clearly positive (in the terminology of \cite{satake:embeddings}, Section 4.3), and it also of rational type: ${\bf m}_\mathbf H$ and ${\bf m}$ are of rational type since $\mathbf H^\der$ and $\mathbf G$ are defined over $\QQ$, from which it follows that ${\bf m}M$ is of rational type (see the proof of \cite{satake:embeddings}, Theorem 3). Therefore, $\Lambda$ is ample, by \cite{satake:embeddings}, Theorem 1, and so too is $L$ (again, see the proof of \cite{satake:embeddings}, Theorem 3).
		
		By definition,
		\begin{align*}
		d\cdot\deg_{L_{K(\mathbf H)}}Y_\mathbf H&=\int_{Y_\mathbf H}c_1(L^{\otimes d}_{K(\mathbf H)})^{\dim Y_\mathbf H}=\int_{Y_\mathbf H}c_1(L\otimes\Lambda)^{\dim Y_\mathbf H}=\\
		&=\sum_{i=0}^{\dim Y_\mathbf H}\binom{\dim Y_\mathbf H}{i}\int_{Y_H}c_1(L)^i\wedge c_1(\Lambda)^{\dim Y_\mathbf H-i}.
		\end{align*}
		Therefore, the result follows from the fact that
		\begin{align*}
		\int_{Y_\mathbf H}c_1(L)^i\wedge c_1(\Lambda)^{\dim Y_\mathbf H-i}\geq 0,
		\end{align*}
		for all $0\leq i\leq\dim Y_\mathbf H-1$, since $L$ and $\Lambda$ are both ample.
	\end{proof}

\section{Choosing measures}\label{measures}

Let $(\mathbf H,X_\mathbf H)$ be a Shimura datum such that $\mathbf H^\der$ is non-trivial (which is to say that any special subvariety associated with $(\mathbf H,X_\mathbf H)$ has positive dimension). Let $\widetilde{\mathbf H}\rightarrow \mathbf H^\der$ denote the simply connected $\QQ$-covering with finite central kernel. Being simply connected, the group $\widetilde{\mathbf H}$ is equal to a direct product $\mathbf H_1\times\cdots\times \mathbf H_s$ of almost $\QQ$-simple, simply connected $\QQ$-groups. By \cite{vasiu:geometry}, Section 3.3, each $\mathbf H_i$ is of the form $\Res_{K_i/\QQ}\mathbf H'_i$ for some totally real field $K_i$ and some absolutely simple, simply connected $K_i$-group $\mathbf H'_i$.

For each $i \in \{ 1, \dots, s \}$ and each archimedean place $v$ of $K_i$, we let $\mu_{i,v}$ denote the Haar measure on $\mathbf H^\prime(K_{i,v})$ induced by the left-invariant exterior form denoted $c_v\omega^*$ in \cite{prasad:volumes}, Section 3.5. Writing $\mu_i:=\prod_{v|\infty}\mu_{i,v}$ we obtain a Haar measure on
\begin{align*}
\mathbf H_i(\RR)=\mathbf H'_i(K_i\otimes_{\QQ}\RR)=\prod_{v|\infty}\mathbf H'(K_{i,v}).
\end{align*}
Hence, if we let $\widetilde{\mu}:=\prod_{i=1}^s\mu_i$, we obtain a Haar measure on $\widetilde{\mathbf H}(\RR)$. We let $\mu$ denote the Haar measure on $\mathbf H^{\der}(\RR)^+$ equal to the pushforward of $\widetilde{\mu}$ under the surjective morphism
\begin{align*}
\widetilde{\mathbf H}(\RR)\rightarrow \mathbf H^\der(\RR)^+
\end{align*}
(for the connectedness of $\widetilde{\mathbf H}(\RR)$ see \cite{milne:intro}, Theorem 5.2 and, for surjectivity, see \cite{milne:intro}, Proposition 5.1).

Let $X^+_\mathbf H$ be a connected component of $X_\mathbf H$ and let $x\in X^+_\mathbf H$. Let $K_\infty$ be the maximal compact subgroup of $\mathbf H^\der(\RR)^+$ equal to the stabilizer of $x$. By \cite{milne:intro}, Lemma 1.5 and Proposition 5.1, we obtain a surjective map
\begin{align*}
\pi_x:\mathbf H^\der(\RR)^+\rightarrow X^+_\mathbf H:h\mapsto hx,
\end{align*}
through which we may identify $X^+_\mathbf H$ with $\mathbf H^\der(\RR)^+/K_\infty$. Since $\mathbf H^\der(\RR)^+$ is unimodular, the pushforward $\pi_{x\ast}\mu$ is independent of the choice of $x$ and, therefore, we also denote it $\mu$. We also denote by $\mu$ the induced measure on any arithmetic quotient of $\mathbf H^\der(\RR)^+$ or $X^+_\mathbf H$. 

On the other hand, for any left $\mathbf H^\der(\RR)^+$-invariant differential form $\nu$ of maximal degree on $X^+_\mathbf H$, we also obtain a measure $\mu_\nu$. We choose such a form $\nu$ such that $\mu_\nu=\mu$. 

Recall that the Lie algebra $\mathfrak{h}$ of $\mathbf H^\der(\RR)^+$ admits a Cartan decomposition $\mathfrak{k}\oplus\mathfrak{p}$, where $\mathfrak{k}$ denotes the Lie algebra of $K_\infty$ and $\mathfrak{p}$ can be identified with the tangent space of $X^+_\mathbf H$ at $K_\infty$. Inside the complexification of $\mathfrak{h}$, we find the Lie algebra $\mathfrak{k}\oplus i\mathfrak{p}$, which corresponds to a maximal compact Lie subgroup $H^c$ of $\mathbf H^\der(\CC)$, which contains $\mathbf H^\der(\RR)^+$. The quotient manifold $H^c/ K_\infty$ is compact and contains $X^+_\mathbf H$ as an open subset. We refer to $H^c/ K_\infty$ as the compact dual of $X^+_\mathbf H$ and denote it $\check{X}_\mathbf H$.

The space $i\mathfrak{p}$ can be identified with the tangent space of $\check{X}_\mathbf H$ at $K_\infty$. The map $\mathfrak{p}\rightarrow i\mathfrak{p}$ sending $p$ to $ip$ induces, therefore, an isomorphism between the tangent spaces of $X^+_\mathbf H$ and $\check{X}_\mathbf H$ at $K_\infty$. In particular, $\nu$ induces a left $H^c$-invariant differential form $\check{\nu}$ of maximal degree on $\check{X}_\mathbf H$ and we denote the corresponding measure $\check{\mu}$. By \cite{prasad:volumes}, Section 3, we have $\check{\mu}(\check{X}_\mathbf H)=1$.

\section{Lower bounds for degrees of a special subvarieties}

Our main ingredient for proving Theorem \ref{maintheorem} is the following generalization of \cite{daw:degrees}, Theorem 1.4. We extend the lower bound given there (for the degree of a strongly special subvariety) to a lower bound for the degree of any positive-dimensional special subvariety.

\begin{teo}\label{mainbound}
Let $(\mathbf G,X)$ be a Shimura datum such that $\mathbf{Z(G)}(\RR)$ is compact and let $X^+$ be a connected component of $X$. Fix a faithful representation $\rho:\mathbf G\rightarrow\GL_n$ and let $K\subset \mathbf G(\AAA_f)$ be a neat compact open subgroup equal to the product of compact open subgroups $K_p\subset \mathbf G(\QQ_p)$. 

There exist effectively computable positive constants $c_1$ and $\delta$ such that, if $V$ is a positive-dimensional special subvariety of $S_K(\mathbf G,X^+)$ defined by $(\mathbf H,X_\mathbf H)$ and $X^+_{\mathbf H}$, then
\begin{align*}
\deg V>c_1\Pi_V^\delta.
\end{align*}
\end{teo}

The purpose of this section is to prove Theorem \ref{mainbound}.

\subsection{Relating the degree to the volume}

As in \cite{daw:degrees}, Theorem 5.1, we first relate the degree of a positive-dimensional special subvariety to its volume. 

\begin{teo}\label{degbound}
Let $(\mathbf G,X)$ be a Shimura datum such that $\mathbf{Z(G)}(\RR)$ is compact and let $X^+$ be a connected component of $X$. Let $K\subset \mathbf G(\AAA_f)$ be a neat compact open subgroup. 

If $V$ is a positive-dimensional special subvariety of $S_K(\mathbf G,X^+)$ defined by $(\mathbf H,X_\mathbf H)$ and $X^+_{\mathbf H}$, then
\begin{align*}
\deg V\geq\mu(\Gamma(\mathbf H)\backslash X^+_\mathbf H).
\end{align*}

\end{teo}


\begin{proof}

By definition, $V$ is equal to the image of the irreducible component $V_\mathbf H:=\Gamma(\mathbf H)\backslash X^+_\mathbf H$ of $\Sh_{K(\mathbf H)}(\mathbf H,X_\mathbf H)$ under the morphism
\begin{align*}
\Sh_{K(\mathbf H)}(\mathbf H,X_\mathbf H)\rightarrow\Sh_K(\mathbf G,X)
\end{align*}
induced from the inclusion $(\mathbf H,X_\mathbf H)\subseteq (\mathbf G,X)$ of Shimura data. Lemma \ref{upperbound} implies the inequality
\begin{align*}
\deg_{L_{K(\mathbf H)}}V_\mathbf H \leq \deg_{L_K}V.
\end{align*}
(Note that we are using the assumption that $\mathbf{Z(G)}(\RR)$ is compact here.)

Consider a smooth compactification $\overline{V_\mathbf H}^{\rm sm}$ of $V_\mathbf H$. As in, for example, the proof of \cite{mumford:prop}, Proposition 3.4 (b), we obtain a canonical birational morphism
\begin{align*}
\pi:\overline{V_\mathbf H}^{\rm sm}\rightarrow\overline{V_\mathbf H},
\end{align*}
where $\overline{V_\mathbf H}$ denotes the Zariski closure of $V_\mathbf H$ in the Baily-Borel compactification of $\Sh_{K(\mathbf H)}(\mathbf H,X_\mathbf H)$.

As mentioned previously, the line bundle $E_0$ of holomorphic forms of maximal degree on $X^+_\mathbf H$ descends to a line bundle $E$ on $V_\mathbf H$, which extends uniquely to an ample line bundle (namely, the restriction of $L_{K(\mathbf H)}$) on $\overline{V_\mathbf H}$. By \cite{mumford:prop}, Proposition 3.4 (b), the pullback $\pi^*L_{K(\mathbf H)}$ of $L_{K(\mathbf H)}$ to $\overline{V_\mathbf H}^{\rm sm}$ is the line bundle $\overline{E}$ afforded to us by \cite{mumford:prop}, Main Theorem 3.1. Furthermore, by the projection formula (see, \cite{KY:AO}, 5.1),
\begin{align*}
\deg_{L_{K(\mathbf H)}}V_\mathbf H=\deg_{\pi^*L_{K(\mathbf H)}}\overline{V_\mathbf H}^{\rm sm}
\end{align*}

Let $c_1(E_0)$ denote the first Chern class of $E_0$. Set $d:=\dim X^+_\mathbf H$ and let $c_1(E_0)^d\in \mathbf H^{2d}(X^+_\mathbf H,\ZZ)$ denote the $d$-fold cup product of $c_1(E_0)$ with itself. We have $c_1(E_0)^d=\lambda [\nu]$ for some $\lambda\in\ZZ$, where $[\nu]\in H^{2d}(X^+_\mathbf H,\ZZ)$ denotes the class of $\nu$. As in the proof of \cite{mumford:prop}, Proportionality Theorem 3.2, we have
\begin{align*}
\deg_{\pi^*L_{K(\mathbf H)}}\overline{V_\mathbf H}^{\rm sm}=\lambda\mu(\mathcal{F})=\lambda\mu(\Gamma(\mathbf H)\backslash X^+_\mathbf H),
\end{align*}
where $\mathcal{F}$ denotes a (Borel) fundamental domain for $\Gamma(\mathbf H)$ in $X^+_\mathbf H$.

On the other hand, let $\check{E}_0$ denote the line bundle of holomorphic forms of maximal degree on $\check{X}_\mathbf H$. Then $c_1(\check{E}_0)^d=\check{\lambda}[\check{\nu}]$ for some $\check{\lambda}\in\ZZ$ and, similarly,
\begin{align*}
\int_{\check{X}_\mathbf H} c_1(\check{E}_0)^d =\check{\lambda}\check{\mu}(\check{X}_\mathbf H).
\end{align*}

However, as explained in \cite{hirzebruch}, Section 2, we have $\lambda=(-1)^d\check{\lambda}$. Therefore, combining the above observations, we obtain
\begin{align*}
\deg_{L_{K(\mathbf H)}}V_\mathbf H=\left[(-1)^d\cdot\int_{\check{X}_\mathbf H} c_1(\check{E}_0)^d\cdot\check{\mu}(\check{X}_\mathbf H)^{-1}\right]\cdot\mu(\Gamma(\mathbf H)\backslash X^+_\mathbf H),
\end{align*}
and the result follows from the fact that $\check{\mu}(\check{X}_\mathbf H)=1$.
\end{proof}

\subsection{Bounding the volume from below}

It remains, then, to give a lower bound for the volume of a positive-dimensional special subvariety. We generalize \cite{daw:degrees}, Theorem 6.1 to all positive-dimensional special subvarieties.

\begin{teo}\label{volbound}
Let $(\mathbf G,X)$ be a Shimura datum such that $\mathbf{Z(G)}(\RR)$ is compact and let $X^+$ be a connected component of $X$. Fix a faithful representation $\rho:\mathbf G\rightarrow\GL_n$ and let $K\subset \mathbf G(\AAA_f)$ be a neat compact open subgroup equal to the product of compact open subgroups $K_p\subset \mathbf G(\QQ_p)$. 

There exist effectively computable positive constants $c_1$ and $\delta$ such that, if $V$ is a positive-dimensional special subvariety of $S_K(\mathbf G,X^+)$ defined by $(\mathbf H,X_\mathbf H)$ and $X^+_{\mathbf H}$, then
\begin{align*}
\mu(\Gamma(\mathbf H)\backslash X^+_\mathbf H)>c_1\Pi_V^\delta.
\end{align*}
\end{teo}

\begin{proof}
We will use the term \textit{uniform} to refer to constants that are independent of $V$. In particular, we seek uniform constants $c_1$ and $\delta$.

As in Section \ref{measures}, let $\widetilde{\mathbf H}\rightarrow \mathbf H^\der$ denote the simply connected $\QQ$-covering with finite central kernel $\mathbf Z$. We will repeatedly need the following lemma. Recall that $|\mathbf Z|:=|\mathbf Z(\CC)|$.

\begin{lem}\label{modZ}
There exists a uniform, effectively computable constant $B$ such that
\begin{align*}
|\mathbf Z|\leq B.
\end{align*}
\end{lem}

\begin{proof}
Recall that $\widetilde{\mathbf H}_\CC$ is the direct product of its almost simple factors. Let $\mathbf F$ be such a factor, and let $D_\mathbf F$ denote its Dynkin diagram. Associated with $D_\mathbf F$ is a corresponding lattice of roots $P(R)$ and a lattice of weights $Q(R)$. By the proof of \cite{uy:andre-oort}, Lemma 2.4, $|\mathbf{Z(F)}|$ is at most $|P(R)/Q(R)|$ and, by the explicit calculations in \cite{bourbaki}, Chapter VI, Section 4, the latter is bounded by $r_\mathbf F+1$, where $r_\mathbf F$ denotes the rank of $\mathbf F$.
\end{proof}

We will also make repeated use of the following lemma. For $n\in\NN$, we denote by $\mu_n$ the algebraic group of $n$-th roots of unity.

\begin{lem}\label{degF}
There exists a uniform, effectively computable constant $b$ and a Galois extension $F/\QQ$ of degree $\leq b$ unramified outside of $\Sigma(\mathbf{H}^\der)$ such that
\begin{align}\label{prodZ}
\mathbf Z_F\cong\prod_i\mu_{n_i,F}.
\end{align}
(In particular, the product is finite and $\prod_in_i=|\mathbf Z|$.)
\end{lem}

\begin{proof}
As $\mathbf{Z}$ is contained in the center of $\widetilde{\mathbf H}$, it is diagonalizable and hence determined by its group of characters $X^\ast(\mathbf{Z})$ (considered as a $\Gal(\overline{\QQ}/\QQ)$-module). In particular, for each field $F \subset \overline{\QQ}$, the base change $\mathbf Z_F$ is of the form \eqref{prodZ} if and only if $\Gal(\overline{\QQ}/F)$ acts trivially on $X^\ast(\mathbf{Z})$. Let $F \subset \overline{\QQ}$ denote the minimal field with this property. Then $F$ is a finite Galois extension of $\QQ$.

In order to bound the degree of $F$, let $\mathbf T$ be a maximal torus of $\widetilde{\mathbf H}$ and take for $F_0$ the minimal splitting field of $\mathbf T$. As $\mathbf{T}$ contains $\mathbf{Z}$, the group $\Gal(\overline{\QQ}/F_0)$ acts trivially on $X^\ast(\mathbf Z)$. Since $\Gal(F_0/\QQ)$ acts faithfully on the character group of $\mathbf T$, it follows that $[F_0:\QQ]$ is bounded by the maximal cardinality of a finite subgroup of $\GL_r(\ZZ)$, where $r$ is the dimension of $\mathbf T$. By Minkowski, the cardinality of such a subgroup is bounded by an effectively computable constant $b(r)$ (see \cite{Serre:Minkowski}). In particular, the bound $b(r_\mathbf G)$ suffices, where $r_\mathbf G$ denotes the rank of $\mathbf G$. Since we have $F \subseteq F_0$, we see that $b(r_\mathbf G)$ bounds the degree of $F$ also.

Now consider a rational prime $p \notin \Sigma(\mathbf{H}^\der)=\Sigma(\widetilde{\mathbf H})$ and an embedding $\iota: \overline{\mathbb Q} \hookrightarrow \overline{\mathbb Q}_p$. (Note that $\iota$ yields an injective homomorphism $\varphi: \Gal(\overline{\mathbb Q}_p/\mathbb{Q}_p)\hookrightarrow \Gal(\overline{\mathbb Q}/\mathbb Q)$.) It suffices to show that $\iota(F) \subset \mathbb{Q}_p^{\mathrm{un}}$, which is equivalent to $\iota(F)$ being fixed by $\Gal(\overline{\QQ}_p/\QQ_p^{\mathrm{un}})$, in other words, $\varphi(\Gal(\overline{\QQ}_p/\QQ_p^{\mathrm{un}}))\subset\Gal(\overline{\mathbb Q}/F)$. Since $p \notin \Sigma(\widetilde{\mathbf{H}})$, there exists a maximal torus $\mathbf T$ of $\widetilde{\mathbf H}_{\QQ_p}$ such that $\mathbf T_{\QQ_p^{\mathrm{un}}}$ splits. As above, this implies that $\Gal(\overline{\mathbb Q}_p/\mathbb Q_p^{\mathrm{un}})$ acts (through $\varphi$) trivially on $X^\ast(\mathbf Z)$, whence the claim.
\end{proof}

By Section \ref{derived}, $\Gamma(\mathbf H)$ is contained in $\mathbf H^\der(\RR)$. Let $\Gamma(\mathbf H)^+$ denote the intersection $\Gamma(\mathbf H)\cap \mathbf H^\der(\RR)^+$. In order to prove Theorem \ref{volbound}, we will need the following lemma.

\begin{lem}
We have 
\begin{align*}
\mu(\Gamma(\mathbf H)\backslash X^+_\mathbf H)\geq\frac{1}{|\mathbf Z|}\cdot\mu(\Gamma(\mathbf H)^+\backslash X^+_\mathbf H).
\end{align*}
\end{lem}

\begin{proof}
Clearly, $\Gamma(\mathbf H)^+$ is a normal subgroup of $\Gamma(\mathbf H)$. Furthermore, the fact that $\Gamma(\mathbf H)$ is neat implies that
\begin{align*}
\Gamma(\mathbf H)^+\backslash X^+_\mathbf H\rightarrow\Gamma(\mathbf H)\backslash X^+_\mathbf H
\end{align*}
is finite \'etale of degree $[\Gamma(\mathbf H):\Gamma(\mathbf H)^+]\leq|\pi_0(\mathbf H^\der(\RR))|$ (see the proof of \cite{uy:andre-oort}, Lemma 2.11, for example). It follows from the proof of \cite{milne:intro}, Corollary 5.3 that 
\begin{align*}
|\pi_0(\mathbf H^\der(\RR))|=H^1(\RR,\mathbf Z_\RR)\leq |\mathbf Z|.
\end{align*} This concludes the proof.
\end{proof}

Let $\mathcal{F}$ denote a Borel fundamental domain in $X^+_\mathbf H$ for $\Gamma(\mathbf H)^+$. By definition, $\mu(\mathcal{F})=\mu(\mathcal{F}^\der)$, where $\mathcal{F}^\der:=\pi^{-1}_x(\mathcal{F})$ is a Borel fundamental domain for $\Gamma(\mathbf H)^+$ in $\mathbf H^\der(\RR)^+$. Let $\Gamma({\widetilde{\mathbf H}})$ denote the preimage of $\Gamma(\mathbf H)^+$ under the induced map
\begin{align*}
\pi:\widetilde{\mathbf H}(\QQ)\rightarrow \mathbf H^{\der}(\QQ)\cap \mathbf H^\der(\RR)^+=\mathbf{H}^{\der}(\QQ)^+
\end{align*}
Note that $\pi$ is not necessarily surjective. Nevertheless, we have
\begin{align*}
\widetilde{\mu}( \Gamma({\widetilde{\mathbf H}})\backslash \widetilde{\mathbf H}(\RR))&=[\mathbf Z(\RR):\mathbf Z(\QQ)]\cdot\mu( \pi(\Gamma({\widetilde{\mathbf H}}))\backslash \mathbf H^\der(\RR)^+)\\
&=[\mathbf Z(\RR):\mathbf Z(\QQ)]\cdot[ \Gamma(\mathbf H)^+:\pi(\Gamma({\widetilde{\mathbf H}}))]\cdot\mu( \Gamma(\mathbf H)^+\backslash \mathbf H^\der(\RR)^+).
\end{align*}

Let $K(\widetilde {\mathbf H})$ denote the preimage of $K(\mathbf H^\der)$ under the map $\widetilde{\mathbf H}(\AAA_f)\rightarrow \mathbf H^\der(\AAA_f)$.
Then $K({\widetilde {\mathbf H}})$ is a compact open subgroup of $\widetilde{\mathbf H}(\AAA_f)$. 
By the assumption of Theorem \ref{volbound}, $K$ is of the form $\prod_p K_p$ with compact open subgroups $K_p \subset \mathbf{G}(\QQ_p)$. From their construction, one easily sees that $K({\mathbf H^\der})$ and $K({\widetilde{\mathbf H}})$ are of the same form. We write $K(\mathbf H^{\der}) = \prod_p K(\mathbf H^{\der})_p$ and $K(\widetilde{\mathbf H}) = \prod_p K(\widetilde{\mathbf H})_p$ with compact open subgroups $K(\mathbf H^{\der})_p \subset \mathbf H^{\der}(\QQ_p)$ and $K(\widetilde{\mathbf H})_p \subset \widetilde{\mathbf H}(\QQ_p)$, respectively. Finally, we let $K({\widetilde{\mathbf H}})^m$ denote a maximal compact open subgroup of $\widetilde{\mathbf H}(\AAA_f)$ containing $K({\widetilde{\mathbf H}})$.

\begin{lem}\label{C-to-power}
There exist uniform, effectively computable positive constants $c_2$ and $C$ such that
\begin{align*}
[\Gamma(\mathbf H)^+:\pi(\Gamma({\widetilde{\mathbf H}}))]<c_2 C^{|\Sigma(\mathbf H^\der,K({\mathbf H^\der}))|}.
\end{align*}

\end{lem}

\begin{proof} The proof is very similar to the proof of \cite{daw:degrees}, Lemma 6.3, and we will refer there for certain details. For each rational prime $p$, Galois cohomology yields a commutative diagram
\begin{equation*}
	\begin{tikzcd} 
	\mathbf{Z}(\QQ) \ar[r]\ar[d] & \widetilde{\mathbf{H}}(\QQ) \ar[r]\ar[d] & \mathbf{H}^\der(\QQ) \ar[r]\ar[d] & H^1(\Gal(\overline{\QQ}/\QQ), \mathbf{Z}(\overline{\QQ})) \ar[d]\\
	\mathbf{Z}(\QQ_p) \ar[r] & \widetilde{\mathbf{H}}(\QQ_p) \ar[r] & \mathbf{H}^\der(\QQ_p) \ar[r] & H^1(\Gal(\overline{\QQ}_p/\QQ_p), \mathbf{Z}(\overline{\QQ}_p))
	\end{tikzcd}
\end{equation*}
with exact rows. As $\Gamma(\widetilde{\mathbf H})$ is by definition the preimage of $\Gamma(\mathbf H)^+$ under $\pi: \widetilde{\mathbf H}(\QQ)\rightarrow \mathbf{H}^{\der}(\QQ)^+$, we deduce a commutative diagram
\begin{equation}
\label{galoiscohomology}
\begin{tikzcd}
\Gamma(\mathbf{H})^+/\pi(\Gamma(\widetilde{\mathbf{H}})) \ar[r] \ar[d] & 
H^1(\Gal(\overline{\QQ}/\QQ), \mathbf{Z}(\overline{\QQ})) \ar[d] \\
K(\mathbf{H}^\der)_p/\mathrm{im}[K(\widetilde{\mathbf{H}})_p \rightarrow K(\mathbf{H}^\der)_p]\ar[r] & 
H^1(\Gal(\overline{\QQ}_p/\QQ_p), \mathbf{Z}(\overline{\QQ}_p)),
\end{tikzcd}
\end{equation}
where every horizontal arrow is an injection. In particular, we can bound the cardinality of $\Gamma(\mathbf{H})^+/\pi(\Gamma(\widetilde{\mathbf{H}}))$ by finding an upper bound on the size of its image $\mathcal{M}$ in $H^1(\Gal(\overline{\QQ}/\QQ), \mathbf{Z}(\overline{\QQ}))$.

Write $\QQ_p^\mathrm{un}$ for the maximal unramified extension of $\QQ_p$ in $\overline{\QQ}_p$ and $\ZZ^\mathrm{un}_p$ for its ring of integers. In the second and third paragraph on p.\ 88 of \cite{daw:degrees}, it is shown that $\mathbf Z$ possesses a smooth model $\mathcal{Z}$ over $\ZZ^\mathrm{un}_p$ and the sequence
\begin{equation*}
K(\widetilde{\mathbf H})_p \rightarrow K(\mathbf H^\der)_p \rightarrow H^1(\Gal(\QQ_p^\un/\QQ_p),\mathcal Z(\ZZ^\mathrm{un}_p))
\end{equation*}
is exact for every rational prime $p \notin \Sigma(\mathbf H^\der,K({\mathbf H^\der}))$ that is additionally coprime to each $n_i$ (as in Lemma \ref{degF}). (Since $\prod_i n_i=|\mathbf Z|$, the latter condition is $(p,|\mathbf Z|)=1$.) We conclude that, for each $\alpha \in \mathcal{M}$ and each such prime $p$, the vertical image $\alpha_p \in H^1(\Gal(\overline{\QQ}_p/\QQ_p), \mathbf{Z}(\overline{\QQ}_p))$ in \eqref{galoiscohomology} is actually contained in $H^1(\Gal(\QQ_p^\un/\QQ_p),\mathcal{Z}(\ZZ^\mathrm{un}_p))$.

We next consider the image of $\mathcal{M}$ under the restriction $$H^1(\Gal(\overline{\QQ}/\QQ),\mathbf Z(\overline{\QQ})) \rightarrow H^1(\Gal(\overline{\QQ}/F),\mathbf Z(\overline{\QQ}))$$ where $F$ is a field as in Lemma \ref{degF}. By the inflation-restriction exact sequence, the kernel of this restriction is $H^1(\Gal(F/\QQ),\mathbf Z(\overline{\QQ}))$. As this is trivially of size $|\mathbf{Z}|^{[F:\QQ]}$, it suffices to bound the image $\mathcal{M}^\prime$ of $\mathcal{M}$ under this restriction. Using Lemma \ref{degF}, we obtain canonical identifications
\begin{equation*}
H^1(\Gal(\overline{\QQ}/F),\mathbf{Z}(\overline{\QQ})) = \prod_{i} F^\times/(F^\times)^{n_i}
\end{equation*}
and
\begin{equation*}
H^1(\Gal(\overline{\QQ}_p/F_\nu),\mathbf{Z}(\overline{\QQ}_p)) = \prod_{i} F^\times_\nu/(F^\times_\nu)^{n_i}.
\end{equation*}
Furthermore, if we let $\Sigma$ denote the union of $\Sigma(\mathbf{H}^\der,K(\mathbf{H}^\der))$ and the set of primes dividing the $n_i$, then, for each place $\nu$ of $F$ lying above a rational prime $p \notin \Sigma$, we have
\begin{equation*}
H^1(\Gal(\QQ_p^\mathrm{un}/F_\nu),\mathcal{Z}(\ZZ^\mathrm{un}_p)) = \prod_{i} \mathcal{O}_{F_\nu}^\times/(\mathcal{O}_{F_\nu}^\times)^{n_i}.
\end{equation*}
(Note that these identifications are also such that the standard homomorphisms on the left correspond to the standard homomorphisms on the right.) With respect to these identifications, our above observations mean that each element $[\beta] = ([\beta_i]) \in \mathcal{M}^\prime \subset \prod_i F^\times /(F^\times)^{n_i}$ is such that $[\beta_{i,\nu}] \in \mathcal{O}_{F_\nu}^\times (F_\nu^\times)^{n_i}/(F_\nu^\times)^{n_i}$ for all places $\nu$ lying over rational primes $p \notin \Sigma$. Hence, for each $i$, we can apply Lemma \ref{lemma::new} below, with $F$, $n_i$, and $\Sigma$ as above. Since $\prod_in_i=|\mathbf Z|$, we have
\begin{align*}
|\Sigma|\leq|\Sigma(\mathbf H^\der, K(\mathbf H^\der))|+(\log|\mathbf{Z}|/\log 2).
\end{align*}
We conclude that
\begin{align*}
|\mathcal{M}| &\leq |\mathbf{Z}|^{[F:\QQ]} |\mathcal{M}^\prime| \leq |\mathbf{Z}|^{[F:\QQ]} \prod_i|\mathrm{im}[\mathcal{M}^\prime\rightarrow F^\times/(F^\times)^{n_i}]|\\
&\leq |\mathbf{Z}|^{[F:\QQ]} \prod_i(2n_i^{[F:\QQ]^2+\Sigma(\mathbf H^\der, K(\mathbf H^\der))|+(\log|\mathbf{Z}|/\log 2)})\\
&\leq|\mathbf{Z}|^{[F:\QQ]+1}  |\mathbf{Z}|^{[F:\QQ]^2+|\Sigma(\mathbf H^\der, K(\mathbf H^\der))|+(\log|\mathbf{Z}|/\log 2)}\\
&=|\mathbf{Z}|^{[F:\QQ]^2+[F:\QQ]+1+(\log|\mathbf{Z}|/\log 2)}|\mathbf{Z}|^{|\Sigma(\mathbf H^\der, K(\mathbf H^\der))|}
\end{align*}
(where we use the observation $\prod_i 2\leq\prod_i n_i=|\mathbf Z|$).
Therefore, letting $B$ and $b$ denote the constants afforded to us by Lemmas \ref{modZ} and \ref{degF}, respectively, we can take
\begin{align*}
c_2=B^{b^2+b+1+2\log B} \text{ and }C=B
\end{align*}
(where we use the estimate $(\log 2)^{-1}\leq 2$).
\end{proof}

The following result was used in the proof of Lemma \ref{C-to-power} above.

\begin{lem}
\label{lemma::new}
Let $F$ be a number field, $n$ a positive integer, and $\Sigma$ a finite set of rational primes such that $F/\QQ$ is unramified outside $\Sigma$. Let $\Sigma_F$ denote the set of all places of $F$ that divide a prime in $\Sigma$ and let $\Sigma_F^c$ denote its complement among all finite places of $F$. Then the subgroup
\begin{equation}
\label{equation::subgroup}
\{ [\beta] \in F^{\times}/(F^{\times})^n\ | \ \forall \nu \in \Sigma_F^c : \beta_\nu \in \mathcal{O}_{F_\nu}^\times (F_\nu^\times)^n/(F_\nu^\times)^n \}
\end{equation}
has cardinality $\leq 2 n^{[F:\QQ]^2+|\Sigma|}$.
\end{lem}

\begin{proof} The inflation-restriction exact sequence
\begin{equation*}
\QQ^\times/(\QQ^\times)^n = H^1(\Gal(\overline{\QQ}/\QQ),\mu_n) \rightarrow F^\times/(F^\times)^n = H^1(\Gal(\overline{\QQ}/F),\mu_n) \rightarrow H^2(\Gal(F/\QQ),\mu_n)
\end{equation*}
and the trivial estimate $H^2(\Gal(F/\QQ),\mu_n) \leq n^{[F:\QQ]^2}$ imply that the cardinality of the subgroup in \eqref{equation::subgroup} is less than $n^{[F:\QQ]^2}$ times the cardinality of
\begin{equation}
\label{equation::subgroup2}
\{ [\beta] \in \QQ^{\times}/(\QQ^{\times})^n\ | \ \forall \nu \in \Sigma_F^c : [\beta_\nu] \in \mathcal{O}_{F_\nu}^\times (F_\nu^\times)^n/(F_\nu^\times)^n \}.
\end{equation}
We claim that the latter group is contained in the image of $\mathcal{O}_{\Sigma,\QQ}^\times$ in $\QQ^\times/(\QQ^\times)^{n}$, where $\mathcal{O}_{\Sigma,\QQ}^\times$ denotes the group of $\Sigma$-units in $\QQ$. It is clear that $\mathcal{O}_{\Sigma,\QQ}^\times$ is isomorphic to $\ZZ/2\ZZ \times \ZZ^{|\Sigma|}$ (see \cite{neukirch}, Corollary I.11.7 for arbitrary number fields). Hence, the lemma follows from the claim.

To prove the claim, we let $\beta \in \QQ^\times$ be such that $[\beta] \in \QQ^\times/(\QQ^\times)^n$ is contained in \eqref{equation::subgroup2}. We may assume $\beta>0$. Write $\beta = p_1^{k_1}p_2^{k_2}\cdots p_r^{k_r}$ with $p_i$ distinct primes and $k_i \neq 0$. It is enough to show that $n | k_i$ if $p_i \notin \Sigma$. Indeed, this would imply $\beta \cdot \prod_{p_i\notin\Sigma} p_i^{-k_i} \in \mathcal{O}_{\Sigma,\QQ}^\times$ and the above claim follows because $[\beta]=[\beta \cdot \prod_{p_i\notin\Sigma} p_i^{-k_i}]$ in $\QQ^\times/(\QQ^\times)^n$.

Therefore, suppose $p_i \notin \Sigma$ and let $\nu \in \Sigma_F^c$ be a prime of $F$ above $p_i$. The ideal $(\beta_\nu) \subset \mathcal{O}_{F_\nu}$ generated by $\beta_\nu$ is then $\mathfrak{p}_\nu^{k_i}$, where $\mathfrak{p}_\nu$ denotes the maximal ideal of $\mathcal{O}_{F_\nu}$ (here we are using the fact that $p$ is unramified in $F$). Since we also have $[\beta_\nu] \in \mathcal{O}_{F_\nu}^\times (F_\nu^\times)^n/(F_\nu^\times)^n$, the integer $k_i$ is necessarily a multiple of $n$. This concludes the proof.
\end{proof}

Now let $\Gamma({\widetilde{\mathbf H}})^m$ denote $\widetilde{\mathbf H}(\QQ)\cap K({\widetilde{\mathbf H}})^m$. Then
\begin{align*}
\widetilde{\mu}( \Gamma({\widetilde{\mathbf H}})\backslash \widetilde{\mathbf H}(\RR))=[\Gamma({\widetilde{\mathbf H}})^m:\Gamma({\widetilde{\mathbf H}})]\cdot\widetilde{\mu}( \Gamma({\widetilde{\mathbf H}})^m\backslash \widetilde{\mathbf H}(\RR))
\end{align*}
and, as explained in the proof of \cite{daw:degrees}, Lemma 6.2 (or simply by applying \cite{milne:intro}, Theorem 4.16),
\begin{align*}
[\Gamma({\widetilde{\mathbf H}})^m:\Gamma({\widetilde{\mathbf H}})]=[ K({\widetilde{\mathbf H}})^m:K({\widetilde{\mathbf H}})].
\end{align*}

From the formula of \cite{prasad:volumes}, it is possible to give a lower bound for the volume of $\Gamma({\widetilde{\mathbf H}})^m\backslash \widetilde{\mathbf H}(\RR)$. Indeed, following the proof of \cite{daw:degrees}, Lemma 6.4, from the sentence, ``Therefore, by [16, Theorem 3.7],...'' on page 10, it is straightforward to verify that
\begin{align*}
\widetilde{\mu}( \Gamma({\widetilde{\mathbf H}})^m\backslash \widetilde{\mathbf H}(\RR))\geq \prod_{i=1}^s\left(\prod_{j=1}^r\frac{m_{i,j}!}{(2\pi)^{m_{i,j}+1}}\right)^{[K_i:\QQ]}\cdot\Pi(\widetilde{\mathbf H},K({\widetilde{\mathbf H}})^m)^{\frac{1}{2}},
\end{align*}
where the $\mathbf{H}'_i$ are as in Section \ref{measures} and the $m_{i,j}$ are the exponents of the simple, simply connected, compact, real analytic Lie group of the same type as the quasi-split inner form of $\mathbf H'_i$. Moreover, by \cite{prasad:volumes}, Section 1.5, we see that
\begin{align*}
\prod_{i=1}^s\left(\prod_{j=1}^r\frac{m_{i,j}!}{(2\pi)^{m_{i,j}+1}}\right)^{[K_i:\QQ]}\geq (2\pi)^{-\dim \mathbf G}.
\end{align*}

Therefore, combining the above, we conclude that $\mu(\Gamma(\mathbf H)\backslash \mathbf H^\der(\RR)^+)$ is bounded from below by
\begin{align}\label{eqn-mid}
(2\pi)^{-\dim \mathbf G}B^{-1}c_2^{-1}C^{-|\Sigma|}\cdot [ K({\widetilde{\mathbf H}})^m:K({\widetilde{\mathbf H}})]\cdot\Pi(\widetilde{\mathbf H},K({\widetilde{\mathbf H}})^m)^{\frac{1}{2}},
\end{align}
where we abbreviate $\Sigma:=\Sigma(\mathbf H^\der,K({\mathbf H^\der}))$. 

We require a replacement for \cite{daw:degrees}, Lemma 6.5.

\begin{lem}\label{char}
Let $\mathbf T\subset \mathbf H^\der_{\QQ_p}$ be a maximal torus. There exists a basis of $X^*(\mathbf T)$ such that the coordinates of the characters of $\mathbf T$ intervening in $\rho$ are bounded in absolute value by a uniform, effectively computable constant $D$.
\end{lem}

\begin{proof}
Consider the root system $\Phi$ associated with $\mathbf H^\der$ and $\mathbf T$ (over $\overline{\QQ}_p$). 
Note that the rank $r$ of $\Phi$ is bounded by the rank of $\mathbf G$. Let $\Delta:=\{\alpha_1,\ldots,\alpha_r\}$ denote a set of simple roots in $\Phi$. The lattice of $X^*(\mathbf T)$ generated by the simple roots is precisely the image of $X^*(\mathbf T^\ad)$, where $\mathbf T^\ad$ is the image of $\mathbf T$ under the finite central map $\mathbf H^\der\rightarrow \mathbf H^\ad$. Therefore, the index of $X^*(\mathbf T^\ad)$ in $X^*(\mathbf T)$ is bounded by the size of the kernel $\mathbf{Z(H^\der)}$, which, by the proof of Lemma \ref{modZ}, is bounded by the $B$ appearing in the statement of Lemma \ref{modZ}.

We can choose a $\ZZ$-basis $\{e_1,\ldots,e_r\}$ of $X^*(\mathbf T)$ and write
\begin{align*}
\alpha_i=\sum_{j=1}^r m_{ji}e_j
\end{align*}
yielding an integral $r\times r$-matrix $M:=(m_{ji})_{ji}$. We multiply $M$ on the left by a unimodular matrix $N$ such that $NM$ is in Hermite normal form. That is, $NM$ is upper triangular and every entry is a non-negative integer bounded by the largest of the diagonal entries. Therefore, since
\begin{align*}
\det(NM)=\det(M)\leq B,
\end{align*}
the latter is a bound on the entries of $NM$. The matrix $NM$ expresses the elements of $\{\alpha_1,\ldots,\alpha_r\}$ in terms of the basis $\{N^{-1}e_1,\ldots,N^{-1}e_r\}$. Therefore, it suffices to bound the absolute value of the coordinates of the characters with respect to the $\QQ$-basis $\Delta$ of $X^*(\mathbf T)_\QQ:=X^*(\mathbf T)\otimes_\ZZ\QQ$.

Since $\mathbf H^\der$ is semisimple, $\rho$ decomposes into a direct sum of irreducible representations of $H^\der_{\overline{\QQ}_p}$. Therefore, we can and do assume that $\rho$ is irreducible.

There exists a character $\chi$ (the heighest weight) of $\mathbf T$ intervening in the restriction of $\rho$ to $\mathbf T$ such that, if we express $\chi$ as a rational expression of simple roots, then the maximum of the absolute values of the coordinates is a bound for the absolute values of the coefficients of the other characters of $\mathbf T$ intervening. That is, it suffices to restrict our attention to $\chi$.

The weight $\chi$ is a dominant weight and, as such, is a non-negative integer linear combination of the fundamental weights $\{w_1,\ldots,w_r\}$. The fundamental weights have the property that
\begin{align}\label{fdwts}
\langle w_i,\alpha_k\rangle= d_i\delta_{ij},
\end{align}
where $\langle\cdot,\cdot\rangle$ is a Euclidean scalar product on $X^*(\mathbf T)_\QQ$ invariant under the Weyl group, and $d_i$ is positive. As linear combinations of the $\alpha_i$, the $w_i$ have positive rational coefficients (given by the inverse of the Cartan matrix). Therefore, since $r$ is constrained by the rank of $G$, it remains to bound the coefficients of $\chi$ as a non-negative integral linear combination of the fundamental weights. This follows from (\ref{fdwts}) and the Weyl dimension formula, which states that
\begin{align*}
n=\dim\rho=\prod_{\alpha\in\Phi^+}\frac{\langle \chi+\phi,\alpha\rangle}{\langle\phi,\alpha\rangle},
\end{align*}
where $\Phi^+$ denotes the set of positive roots in $\Phi$ and $\phi$ is half the sum of the positive roots.
\end{proof}

We require the following version of \cite{daw:degrees}, Lemma 6.6.

\begin{lem}\label{new66}
There exist uniform, effectively computable constants $c_3$ and $c_4$ such that, for any $p\notin\Sigma(\widetilde{\mathbf H},K(\widetilde{\mathbf H})^m)$ greater than $c_3$ such that $K(\widetilde{\mathbf H})_p\subsetneq K(\widetilde{\mathbf H})^m_p$, we have
\begin{align*}
[K(\widetilde{\mathbf H})^m_p:K(\widetilde{\mathbf H})_p]>c_4p.
\end{align*}
\end{lem}

\begin{proof}
The proof proceeds as in \cite{daw:degrees}, Lemma 6.6, after replacing $\mathbf H$ with $\mathbf H^\der$ (and adjusting to our notations).

First, note that we have to restrict to primes $p$ such that ${\mathcal Z}_{\FF_p}$ is smooth (where $\mathcal Z$ is the kernel of the unique extension over $\ZZ_p$ of the natural map $\widetilde{\bf H}_{\QQ_p}\rightarrow \mathbf H^\der_{\QQ_p}$) and $K_p={\mathbf G}_{\ZZ_p}(\ZZ_p)$. Taking $p>\max\{|\mathbf Z|,p_0\}$, where $p_0$ is the largest prime such that $K_p\neq {\mathbf G}_{\ZZ_p}(\ZZ_p)$, suffices for this purpose. Then, for such a prime $p$, such that $p\notin\Sigma(\widetilde{\mathbf H},K(\widetilde{\mathbf H})^m)$ and such that $K(\widetilde{\mathbf H})_p\subsetneq K(\widetilde{\mathbf H})^m_p$, it is necessary to bound $|H^1(\Gal(\QQ^{\rm un}_p/\QQ_p),{\mathcal Z}(\ZZ^{\rm un}_p))|$. 

From the exact sequence
\begin{align*}
0\rightarrow H^1(\Gal(F_v/\QQ_p),\mathbf Z(F_v))\rightarrow H^1(\Gal(\QQ^{\rm un}_p/\QQ_p),{\mathcal Z}(\ZZ^{\rm un}_p))\rightarrow\\ \rightarrow H^1(\Gal(\QQ^{\rm un}_p/F_v),{\mathcal Z}(\ZZ^{\rm un}_p)),
\end{align*}
where $F$ is the field afforded to us by Lemma \ref{degF} and $v$ is a place of $F$ lying above $p$, 
we obtain the bound $|\mathbf Z|^{[F:\QQ]}|H^1(\Gal(\QQ^{\rm un}_p/F_v),{\mathcal Z}(\ZZ^{\rm un}_p))|$. 

For $n\in\NN$, the group $H^1(\Gal(\QQ^{\rm un}_p/F_v),\mu_n(\ZZ^{\rm un}_p))$ is isomorphic to $\mathcal{O}^\times_{F_v}/(\mathcal{O}_{F_v}^\times)^n$. Furthermore, by \cite{neukirch}, Proposition 5.7 (i),
\begin{align*}
\mathcal{O}^\times_{F_v}=\mu_{q-1}\oplus(1+\mathfrak{p}_\nu)=\ZZ/(q-1)\ZZ\oplus\ZZ/p^a\ZZ\oplus\ZZ^d_p,
\end{align*}	
where $d:=[F_v:\QQ_p]$, $\mathfrak{p}_\nu$ is the maximal ideal of $\mathcal{O}_{F_v}$, $q=|\mathcal{O}_{F_v}/\mathfrak{p}_\nu|$, and $a\geq 0$. It follows that 
\begin{align*}
|H^1(\Gal(\QQ^{\rm un}_p/F_v),\mu_n(\ZZ^{\rm un}_p))|\leq n^{[F:\QQ]+2}
\end{align*}
and so
\begin{align*}
|H^1(\Gal(\QQ^{\rm un}_p/\QQ_p),{\mathcal Z}(\ZZ^{\rm un}_p))|&\leq |\mathbf Z|^{[F:\QQ]}|H^1(\Gal(\QQ^{\rm un}_p/F_v),{\mathcal Z}(\ZZ^{\rm un}_p))|\\
&=|\mathbf Z|^{[F:\QQ]}|\prod_iH^1(\Gal(\QQ^{\rm un}_p/F_v),\mu_{n_i}(\ZZ^{\rm un}_p))|\\
&\leq |\mathbf Z|^{[F:\QQ]}|\prod_in_i^{[F:\QQ]+2}=|\mathbf Z|^{2[F:\QQ]+2}\leq B^{2b+2},
\end{align*}
where $B$ and $b$ are the uniform constants afforded to us from Lemmas \ref{modZ} and \ref{degF}, respectively.

Let $\mathbf T_0$ denote the maximal torus of $\mathbf H^\der_{\QQ_p}$ occuring in the proof of \cite{daw:degrees}, Lemma 6.6, and let $r$ denote its dimension. As explained in the proof of \cite{daw:degrees}, Lemma 6.6, $\mathbf T_{0,\ZZ_p}$ is a torus and we have a canonical isomorphism 
\begin{align*}
X^*(\mathbf T_{0,\overline{\QQ}_p})\rightarrow X^*(\mathbf T_{0,\overline{\FF}_p})
\end{align*}
identifying the characters intervening in $\rho_{\overline{\QQ}_p}$ and $\rho_{\overline{\FF}_p}	$. Therefore, if we let $D$ denote the constant afforded to us by Lemma \ref{char}, we obtain a basis of $X^*(\mathbf T_{0,\overline{\FF}_p})\cong\ZZ^r$ such that the characters intervening in $\rho_{\overline{\FF}_p}$ have coordinates of absolute value at most $D$. 

Following the proof of \cite{EY:subvar}, Proposition 4.3.9, it remains to consider subgroups of $\mathbf T_{0,\overline{\FF}_p}$ of the form $\mathbf T_X:=\cap_{\chi\in X}\ker\chi$, varying over subsets $X$ of $X^*(\mathbf T_{0,\overline{\FF}_p})$ whose members have coordinates of absolute value at most $2D$, and to give a bound on the size of their group of connected components $\pi_0(\mathbf T_X)$. Let $X$ denote such a subset and let $M_X$ denote the matrix whose rows express the coordinates of the elements of $X$ with respect to the basis above. Putting $M_X$ into Smith Normal Form, we obtain an isomorphism
\begin{align*}
X^*(\mathbf T_X)=\ZZ^{r_X}\oplus\ZZ/d_1\ZZ\oplus\cdots\oplus\ZZ/d_X\ZZ,
\end{align*}
where $d_1|d_2|\cdots |d_X$ are the non-zero diagnonal entries of the Smith Normal Form. Since $\pi_0(\mathbf T_X)$ is of size $d_1d_2\cdots d_X$, and the absolute values of the coordinates of the members of $X$ are uniformly bounded, we obtain a uniform and effectively computable bound $c$ on $|\pi_0(\mathbf T_X)|$. It now follows from our previous calculations and \cite{EY:subvar}, Section 4.4, that
\begin{align*}
[K(\widetilde{\mathbf H})^m_p:K(\widetilde{\mathbf H})_p]>(cB^{2b+2})^{-1}(p-1).
\end{align*}
\end{proof}

We require the following version of \cite{daw:degrees}, Lemma 6.7.

\begin{lem}\label{new67}
There exists a uniform, effectively computable constant $c_5$ such that, if $p\notin\Sigma(\widetilde{\mathbf H},K(\widetilde{\mathbf H}))$ is a prime greater than $c_5$, then $p\notin\Sigma(\mathbf H,K(\mathbf H))$.
\end{lem}

\begin{proof}
Again, the proof proceeds as in \cite{daw:degrees}, Lemma 6.7, after replacing $\mathbf H$ with $\mathbf H^\der$ (and adjusting to our notations), and we conclude that we can take 
\begin{align*}
c_5=\max\{c_3,c^{-1}_4B^{2b+2}\},
\end{align*}
where $B$ and $b$ are the uniform constants afforded to us from Lemmas \ref{modZ} and \ref{degF}, respectively, and $c_3$ and $c_4$ are the uniform constants afforded to us by Lemma \ref{new66}.
\end{proof}

The proof of Theorem \ref{volbound} now concludes as in \cite{daw:degrees}, combining the lower bound (\ref{eqn-mid}) with Lemmas \ref{new66} and \ref{new67}. 
\end{proof}

\begin{proof}[Proof of Theorem \ref{mainbound}]
This follows by combining Theorems \ref{degbound} and \ref{volbound}.
\end{proof}

\section{Hecke correspondences}

In this section, we generalise \cite{daw:degrees}, Theorem 7.1. 

If $K$ is a compact open subgroup of $\mathbf G(\AAA_f)$ equal to a product of compact open subgroups $K_p\subset \mathbf G(\QQ)_p$, we will use the notation $K^p$ to denote the product $\prod_{l\neq p}K_l$.

\begin{teo}\label{hecke}
Let $(\mathbf G,X)$ be a Shimura datum such that $\mathbf G=\mathbf G^{\ad}$ and let $X^+$ be a connected component of $X$. Fix a faithful representation $\rho:\mathbf G\rightarrow\GL_n$ and let $K$ be a neat compact open subgroup of $\mathbf G(\AAA_f)$ equal to a product of compact open subgroups $K_p\subset \mathbf G(\QQ_p)$.
 
There exist effectively computable positive integers $k$ and $f$ such that, if 
\begin{itemize}
\item $(\mathbf P,X_\mathbf P)$ is a Shimura subdatum of $(\mathbf G,X)$ and $X^+_\mathbf P$ is a connected component of $X_\mathbf P$ contained in $X^+$,
\item $V$ is a positive-dimensional special subvariety of $S_{K(\mathbf P)}(\mathbf P,X_\mathbf P^+)$ defined by $(\mathbf H,X_\mathbf H)$ and $X^+_{\mathbf H}$, and
\item $p\notin\Sigma_V$ is a prime such that $K_p=\mathbf G_{\ZZ_p}(\ZZ_p)$, 
\end{itemize}
then there exists a compact open subgroup $I(\mathbf P)_p$ of $\mathbf P(\QQ_p)$ contained in $K(\mathbf P)_p:=K_p\cap \mathbf P(\QQ_p)$ and an element $\alpha\in \mathbf P(\QQ_p)$ such that
\begin{itemize}
\item[(1)] $[K(\mathbf P)_p:I(\mathbf P)_p]\leq p^f$,

\medskip

\item[(2)] $[I(\mathbf P)_p:I(\mathbf P)_p\cap\alpha I(\mathbf P)_p\alpha^{-1}]\leq p^k$,

\medskip

\item[(3)] if $I(\mathbf P):=K(\mathbf P)^pI(\mathbf P)_p\subset \mathbf P(\AAA_f)$,
\begin{align*}
\pi_{I(\mathbf P),K(\mathbf P)}:S_{I(\mathbf P)}(\mathbf P,X_\mathbf P^+)\rightarrow S_{K(\mathbf P)}(\mathbf P,X_\mathbf P^+)
\end{align*}

denotes the natural morphism, and $\widetilde{V}$ is an irreducible component of $\pi_{I(\mathbf P),K(\mathbf P)}^{-1}(V)$, then $\widetilde{V}\subseteq T_{\alpha}(\widetilde{V})$, and

\medskip

\item[(4)] if $\mathbf P^\ad=\mathbf P_1\times\cdots\times \mathbf P_n$ is the decomposition of $\mathbf P^\ad$ into the product of its $\QQ$-simple factors, and $I\subseteq\{1,\ldots,n\}$ is the set of those $i$ for which the natural projection $\pi_i:\mathbf P\rightarrow \mathbf \mathbf \mathbf P_i$ restricts non-trivially to $\mathbf H^\der$, then $\pi_i(k_1\alpha k_2)$ generates an unbounded subgroup of $\mathbf P_i(\QQ_p)$ for every $i\in I$ and for all $k_1,k_2\in I(\mathbf P)_p$.
\end{itemize}
\end{teo}

\begin{rem}
Note that, if $V$ is non-facteur, then, by Lemma \ref{nonfacnontriv}, $I=\{1,\ldots,n\}$. 
\end{rem}

We will use the term uniform to refer to constants depending only on the data in the first paragraph of the statement of Theorem \ref{hecke}. We will deal first with the matter of including a positive-dimensional special subvariety in its image under a Hecke correspondence.

\begin{lem}\label{inc}
There exists a uniform, effectively computable positive integer $A$ such that, for any $\alpha\in \mathbf H^\der(\AAA_f)$,
\begin{align}V\subseteq T_{\alpha^A}(V).\nonumber\end{align}
\end{lem}

\begin{proof}
By definition, $V$ is the image of $X^+_\mathbf H\times\{1\}$ in $\Sh_K(\mathbf G,X)$. Thus, consider a point $\overline{(x,1)}\in V$ with $x\in X^+_\mathbf H$. Let
\begin{align}\pi:\widetilde{\mathbf H}\rightarrow \mathbf H^\der\nonumber\end{align}
be the simply connected covering, whose degree we denote $d$, and consider an $\alpha\in \mathbf H^\der(\AAA_f)$. Therefore, for any positive integer $A$ divisible by $d$, there exists $\beta\in\widetilde{\mathbf H}(\AAA_f)$ such that $\pi(\beta)=\alpha^A$. By strong approximation applied to $\widetilde{\mathbf H}$, we have $\beta=qk$, for some $q\in\widetilde{\mathbf H}(\QQ)$ and some $k\in\pi^{-1}(K(\mathbf H^\der))$. Note that, since $\pi$ is proper, $\pi^{-1}(K(\mathbf H^\der))$ is a compact open subgroup of $\widetilde{\mathbf H}(\AAA_f)$. Since $\widetilde{\mathbf H}(\RR)$ is connected, $\pi(q)\in \mathbf H^\der(\RR)^+$ and $\pi(q)\cdot x\in X^+_\mathbf H$. 

Thus, consider the point
\begin{align}\overline{(\pi(q)\cdot x,\pi(\beta))}\in T_{\alpha^A}(V).\nonumber\end{align} 
By the previous discussion, this is equal to $\overline{(x,1)}$. Therefore, setting $A=B!$, with $B$ the constant afforded to us by Lemma \ref{modZ}, finishes the proof.
\end{proof}

In order to find suitable Hecke correspondences, we will also need the following two results on maximal split tori.

\begin{lem}\label{tor}
Let $p\notin\Sigma(\mathbf H^\der,K(\mathbf H^\der))$ be a prime such that $K_p=\mathbf G_{\ZZ_p}(\ZZ_p)$. Then there exists a maximal split torus $\mathbf S\subset \mathbf H^\der_{\QQ_p}$ such that $\mathbf S_{\ZZ_p}$ is a torus.
\end{lem}

\begin{proof}
The proof is identical to \cite{daw:degrees}, Lemma 7.3, after replacing $\mathbf H$ with $\mathbf H^\der$ (and adjusting to our notations).
\end{proof}

\begin{lem}\label{split-char}
Assume $\mathbf H^\der_{\QQ_p}$ is quasi-split and let $\mathbf S\subset \mathbf H^\der_{\QQ_p}$ be a maximal split torus. There exists a basis of $X^*(\mathbf S)$ such that the coordinates of the characters of $\mathbf S$ that intervene in $\rho$ are bounded in absolute value by a uniform, effectively computable positive constant $D_{\rm sp}$.
\end{lem}

\begin{proof}
Let $\mathbf T$ denote the centraliser of $\mathbf S$ in $\mathbf H^\der_{\QQ_p}$. Since $\mathbf H^\der_{\QQ_p}$ is quasi-split, $\mathbf T$ is a maximal torus of $\mathbf H^\der_{\QQ_p}$. Let $\mathbf A$ denote the maximal $\QQ_p$-anisotropic subtorus of $\mathbf T$. By the proof of \cite{daw:degrees}, Lemma 7.4 (after replacing $\mathbf H$ with $\mathbf H^\der$ and adjusting to our notations), there exists an embedding
\begin{align*}
\varphi:L:=X^*(\mathbf S)\oplus X^*(\mathbf A)\rightarrow X^*(\mathbf T)
\end{align*}
such that, with respect to a basis $\{e_1,\ldots,e_r\}$ of $X^*(\mathbf T)$, the images of the characters of $S$ intervening in $\rho$ have coefficients of absolute value at most $b^{r_{\mathbf G}}D$, where $r_{\mathbf G}$ is the rank of $\mathbf G$, $b$ is the constant afforded to us by Lemma \ref{degF}, and $D$ is the constant afforded to us by Lemma \ref{char}.

The multiplication map $\mathbf S\times \mathbf A\rightarrow \mathbf T$ induces an embedding
\begin{align*}
\phi:X^*(\mathbf T)\rightarrow L=X^*(\mathbf S)\oplus X^*(\mathbf A)
\end{align*}
Choose a basis $\{f_1,...,f_s\}$ for $X^*(\mathbf S)$ and a basis $\{f_{s+1},\ldots,f_r\}$ for $X^*(\mathbf A)$. These combine to form a basis $\mathcal{B}$ of $L$ and we can write the $\phi(e_i)$ in terms of $\mathcal{B}$ to yield an $r\times r$ matrix $M$. As in the proof of Lemma \ref{char}, we can multiply $M$ on the left by a unimodular matrix $N$ such that $NM$ is in Hermite normal form. Since $L$ contains the image of $\phi$ and the image of $\varphi\circ\phi$ has index at most $b^{r_{\mathbf G}}D$, we conclude that the coefficients of $NM$ are bounded by $b^{r_{\mathbf G}}D$. That is, the absolute values of the coefficients of the $\phi(e_i)$ with respect to $N^{-1}\mathcal{B}$ are at most $b^{r_{\mathbf G}}D$.

It follows that the absolute values of the coefficients of the images under $\phi\circ\varphi$ of the characters of $S$ intervening in $\rho$ with respect to $\{N^{-1}f_1,\ldots,N^{-1}f_s\}$ are at most $r_{\mathbf G}b^{2r_{\mathbf G}}D^2$, and hence the same is true of the characters themselves.
\end{proof}

\begin{proof}[Proof of Theorem \ref{hecke}]
By Lemma \ref{tor}, since $p\notin\Sigma_V$, we can find a non-trivial maximal split torus $\mathbf S\subset \mathbf H^\der_{\QQ_p}$ such that $\mathbf S_{\ZZ_p}$ is a torus. Furthermore, by Lemma \ref{char}, there exists a basis of $X^*(\mathbf S)$ such that the coordinates of the characters intervening in $\rho$ are bounded in absolute value by a uniform, effectively computable constant $D_{\rm sp}$. We need the following Lemma (omitted in \cite{daw:degrees}).

\begin{lem}\label{nontriv}
For every $i\in I$, the split torus $\pi_i(\mathbf S)$ is non-trivial.
\end{lem}

\begin{proof}
Let $\rho:\wt{\mathbf H}\rightarrow \mathbf H^\der$ denote the simply connected cover and let $\wt{\mathbf S}$ denote a maximal split torus of $\wt{\mathbf H}_{\QQ_p}$ such that $\rho(\wt{\mathbf S})=\mathbf S$. We can write $\wt{\mathbf S}$ as a product of maximal split tori $\mathbf S_j$ in the $\QQ_p$-almost-simple factors $\mathbf H_j$ of $\wt{\mathbf H}_{\QQ_p}$. 

The map $\rho_{\QQ_p}$ composed with the inclusion of $\mathbf H^\der_{\QQ_p}$ in the product of the $\pi_i(\mathbf H^\der)_{\QQ_p}$ is given by maps $f_i$, each a product of morphisms $g_{i,j}:\mathbf H_j\rightarrow\pi_i(\mathbf H^\der)_{\QQ_p}$. If $i\in I$, then one of the $g_{i,j}$ must be non-trivial. Since $\mathbf H_j$ is almost-simple, $\ker g_{i,j}$ is finite and, therefore, $g_{i,j}(\mathbf S_j)$ is non-trivial.
\end{proof}

The proof of Theorem \ref{hecke} now proceeds identically to \cite{daw:degrees}, Theorem 7.1, restricting to the $\pi_i$ for $i\in I$. By the proof of \cite{ey:subvarieties}, Lemma 7.4.3, we see that we can take
\begin{align*}
k'=An^3(2D_{\rm sp})r_\mathbf G,
\end{align*}
where $A$ denotes the constant afforded to us by Lemma \ref{inc} and $r_\mathbf G$ denotes the rank of $\mathbf G$.
Now let $W$ denote the finite Weyl group of $\mathbf H^\der_{\QQ_p}$, let $\Phi^+$ denote the set of positive roots, and let $d$ denote the maximum of the values associated to the vertices of the local Dynkin diagram. We see from the proof of \cite{KY:AO}, Lemma 8.1.6 (b) that it suffices to find $f$ such that
\begin{align*}
p^f>|W|^2p^{d|\Phi^+|}.
\end{align*}
This is possible given that $|W|$, $d$, and $|\Phi^+|$ are themselves all uniformly bounded by an effectively computable constant. Finally, as in the proof of \cite{daw:degrees}, Theorem 7.1, we let $k=k'+f$.
\end{proof}

\section{The geometric criterion}

In this section, we give an almost immediate generalisation of \cite{daw:degrees}, Theorem 8.1.

\begin{teo}\label{criterion}
Let $(\mathbf G,X)$ be a Shimura datum and let $\mathbf G^\ad=\mathbf G_1\times\cdots\times \mathbf G_n$ denote the decomposition of $\mathbf G^\ad$ into the product of its $\QQ$-simple factors. Let $X^+$ be a connected component of $X$ and let $K$ be a neat compact open subgroup of $\mathbf G(\AAA_f)$ equal to a product of compact open subgroups $K_p\subset \mathbf G(\QQ_p)$. 

Let $V$ be a positive-dimensional special subvariety of $S_K(\mathbf G,X^+)$ defined by $(\mathbf H,X_\mathbf H)$ and $X^+_{\mathbf H}$ and contained in a Hodge generic subvariety $Z$ of $S_K(\mathbf G,X^+)$. Let $I\subseteq\{1,\ldots,n\}$ denote the set of those $i$ for which the natural projection $\pi_i:\mathbf G\rightarrow \mathbf G_i$ restricts non-trivially to $\mathbf H^\der$. Suppose that there exists a prime $p$ and an $\alpha\in \mathbf G(\QQ_p)$ such that
\begin{itemize}
\item $Z\subseteq T_\alpha(Z)$ and
\item $\pi_i(k_1\alpha k_2)$ generates an unbounded subgroup of $\mathbf G_i(\QQ_p)$ for every $i\in I$ and for all $k_1,k_2\in K_p$.
\end{itemize}
Then, either
\begin{itemize}
\item[(1)] $Z$ contains a special subvariety $V'$ of $S_K(\mathbf G,X^+)$ such that $V\subsetneq V'$, or
\item[(2)] the image of $\mathbf H^\der$ in $\mathbf G^\ad$ is equal to $\prod_{i\in I}\mathbf G_i$.
\end{itemize}
\end{teo}

\begin{rem}
Note that, by Lemma \ref{nonfacnontriv}, if $V$ is non-facteur, then the second conclusion can only hold if $V=Z=S_K(\mathbf G,X^+)$.
\end{rem}

\begin{proof}[Proof of Theorem \ref{criterion}]
The proof is almost identical to that of \cite{daw:degrees}, Theorem 8.1. We must simply account for the fact that the group $\pi_i(U_p)$ defined in the proof of \cite{daw:degrees}, Theorem 8.1 is only unbounded for $i\in I$. Therefore, restricting to those $i\in I$, we obtain the first conclusion of the theorem, unless, for all $i\in I$, the image of $\mathbf H^\der$ in $\mathbf G^\ad$ is equal to $\mathbf G_i\times \mathbf H^\der_{\neq i}$, where $\mathbf H^\der_{\neq i}$ denotes the projection of $\mathbf H^\der$ to $\prod_{j\in I\setminus\{i\}}\mathbf G_i$. From this, the second conclusion follows immediately. 
\end{proof}

\section{Assumptions and reductions for Theorem \ref{maintheorem}}\label{statement}

The assumptions alluded to in Theorem \ref{maintheorem} are as follows:
\begin{itemize}
	\item[(1)] $\mathbf G=\mathbf G^\ad$;
	\item[(2)] $K$ is equal to a product of compact open subgroups $K_p\subset \mathbf G(\QQ_p)$.
\end{itemize}

In this section, we will show that it suffices to prove Theorem \ref{maintheorem} under the following additional assumptions:
\begin{itemize}
	\item[(3)] $Z$ is contained in a connected component $S_K(\mathbf G, X^+)$ of $\Sh_K(\mathbf G, X)$ for some fixed connected component $X^+$ of $X$;
	\item[(4)] the compact open subgroup $K\subset \mathbf G(\AAA_f)$ is neat;
\end{itemize}

\subsection{Working in a connected component of $S$}

Let $X^+$ denote a connected component of $X$. Since $Z$ is irreducible, it belongs to an irreducible component of $\Sh_K(\mathbf G,X)$. That is, there exists $\alpha\in \mathbf G(\AAA_f)$ such that $Z$ is contained in the image of $X^+\times\{\alpha\}$ in $\Sh_K(\mathbf G,X)$ (by the proof of \cite{milne:intro}, Lemma 5.11, every element of $X$ is of the form $qx$, where $q\in \mathbf G(\QQ)$ and $x\in X^+$). Therefore, since the isomorphism
\begin{align*}
[\alpha^{-1}]:\Sh_K(\mathbf G,X)\rightarrow\Sh_{\alpha K\alpha^{-1}}(\mathbf G,X)
\end{align*}
preserves degrees and the property of being a non-facteur maximal special subvariety, we can and do assume that $Z$ is contained in $S_K(\mathbf G,X^+)$. This justifies our assumption (3).

\subsection{Assuming that $K$ is neat}

By the paragraphs preceding \cite{uy:andre-oort}, Lemma 2.11, $K$ contains a neat compact open subgroup $K'$ of $G(\AAA_f)$ whose index in $K$ is at most $|\GL_n(\FF_3)|$. We obtain a finite morphism
\begin{align*}
\pi_{K',K}:\Sh_{K'}(\mathbf G,X)\rightarrow\Sh_K(\mathbf G,X),
\end{align*}
of degree $[K:K']$ and, by \cite{KY:AO}, Proposition 5.3.2 (1), we have $\pi_{K',K}^*L_K=L_{K'}$. Therefore, if we let $Z'$ denote an irreducible component of $\pi_{K',K}^{-1}(Z)$, then, by the projection formula, we have
\begin{align*}
\deg_{L_{K'}}Z'\leq[K:K']\deg_{L_K}Z.
\end{align*}
Furthermore, if $V$ is a non-facteur maximal special subvariety of $Z$, then some irreducible component $V'$ of $\pi_{K',K}^{-1}(V)$ is a non-facteur maximal special subvariety of $Z'$ and
\begin{align*}
\deg_{L_K}V\leq \deg_{L_{K'}}V'.
\end{align*}
We conclude that we lose no generality in Theorem \ref{maintheorem} if we assume that $K$ is neat.

\section{Proof of Main Theorem \ref{maintheorem}}

Finally, we are in a position to combine the various tools already established. First we prove an inductive step.

\begin{prop}\label{induction}
Let $(\mathbf G,X)$ be a Shimura datum such that $\mathbf G=\mathbf G^\ad$ and let $X^+$ be a connected component of $X$. Fix a faithful representation $\rho:\mathbf G\rightarrow\GL_n$ and let $K\subset \mathbf G(\AAA_f)$ be a neat compact open subgroup equal to the product of compact open subgroups $K_p\subset \mathbf G(\QQ_p)$. Let $k$ and $f$ be the effectively computable, positive integers afforded to us by Theorem \ref{hecke}.

Let $(\mathbf P,X_\mathbf P)$ be a Shimura subdatum and let $X^+_\mathbf P$ denote a connected component of $X_\mathbf P$ contained in $X^+$. Let $Z$ be a Hodge generic, proper subvariety of $S_{K(\mathbf P)}(\mathbf P,X_\mathbf P^+)$ and let $V$ be a maximal special subvariety of $Z$ of positive dimension defined by $(\mathbf H,X_\mathbf H)$ and $X^+_{\mathbf H}$. Then, either 
\begin{itemize}
\item the image of $\mathbf H^\der$ in $\mathbf P^\ad$ is equal to a product of $\QQ$-simple factors, or
\item for any prime $p\notin\Sigma_{V}$ such that $K_p=\mathbf G_{\ZZ_p}(\ZZ_p)$, there exists an irreducible subvariety $Y\subsetneq Z$ containing $V$ such that
\begin{align*}
\deg_{L_{K(\mathbf P)}}(Y)\leq p^{2f+k}\deg_{L_{K(\mathbf P)}}(Z)^2.
\end{align*}
\end{itemize}
\end{prop}

\begin{proof}
Let $\mathbf P^\ad=\mathbf P_1\times\cdots\times \mathbf P_n$ denote the decomposition of $\mathbf P^\ad$ into the product of its $\QQ$-simple factors and let $I\subseteq\{1,\ldots,n\}$ denote the set of those $i$ for which the natural projection $\pi_i:\mathbf P\rightarrow \mathbf P_i$ restricts non-trivially to $\mathbf H^\der$.

By Theorem \ref{hecke}, for any prime $p$ as in the statement of the theorem, there exists a compact open subgroup $I(\mathbf P)_p\subset \mathbf P(\QQ_p)$ contained in $K(\mathbf P)_p:=K_p\cap \mathbf P(\QQ_p)$ and an $\alpha\in\mathbf P(\QQ_p)$ such that
\begin{itemize}
\item[(1)] $[K(\mathbf P)_p:I(\mathbf P)_p]\leq p^f$,

\medskip

\item[(2)] $[I(\mathbf P)_p:I(\mathbf P)_p\cap\alpha I(\mathbf P)_p\alpha^{-1}]\leq p^k$,

\medskip

\item[(3)] if $I(\mathbf P):=K(\mathbf P)^pI(\mathbf P)_p\subset \mathbf P(\AAA_f)$,
\begin{align*}
\pi_{I(\mathbf P),K(\mathbf P)}:S_{I(\mathbf P)}(\mathbf P,X_\mathbf P^+)\rightarrow S_{K(\mathbf P)}(\mathbf P,X_\mathbf P^+)
\end{align*}

denotes the natural morphism, and $\widetilde{V}$ is an irreducible component of $\pi_{I(\mathbf P),K(\mathbf P)}^{-1}(V)$, then $\widetilde{V}\subseteq T_{\alpha}(\widetilde{V})$, and

\medskip

\item[(4)] $\pi_i(k_1\alpha k_2)$ generates an unbounded subgroup of $\mathbf P_i(\QQ_p)$ for every $i\in I$ and for all $k_1,k_2\in I(\mathbf P)_p$.
\end{itemize}

Let $\widetilde{Z}$ be an irreducible component of $\pi_{I(\mathbf P),K(\mathbf P)}^{-1}(Z)$ containing $\widetilde{V}$. Since $V$ is a maximal special subvariety of $Z$, $\widetilde{V}$ is a maximal special subvariety of $\widetilde{Z}$. Therefore, by Theorem \ref{criterion}, either the second conclusion of the theorem holds, and the proof is finished, or we can eliminate the possibility that $\widetilde{Z}$ is contained in $T_{\alpha}(\widetilde{Z})$. 

In the latter case, the irreducible components of $\widetilde{Z}\cap T_{\alpha}(\widetilde{Z})$ are strictly contained in $\widetilde{Z}$. By (3) above, we also have 
\begin{align*}
\widetilde{V}\subseteq \widetilde{Z}\cap T_{\alpha}(\widetilde{Z}).
\end{align*}
Therefore, let $\widetilde{Y}$ denote an irreducible component of the $\widetilde{Z}\cap T_{\alpha}(\widetilde{Z})$ containing $\widetilde{V}$, and let $Y:=\pi_{I(\mathbf P),K(\mathbf P)}(\widetilde{Y})$. We have
\begin{align*}
\deg_{L_{K(\mathbf P)}}Y&\leq\deg_{L_{I(\mathbf P)}}\widetilde{Y}\\
&\leq(\deg_{L_{I(\mathbf P)}}\widetilde{Z})^2\cdot[I(\mathbf P)_p:I(\mathbf P)_p\cap\alpha I(\mathbf P)_p\alpha^{-1}]\\
&\leq p^{2f+k}\cdot(\deg_{L_{K(\mathbf P)}}Z )^2,
\end{align*}
where the first inequality follows from the projection formula, the second follows from Bezout's theorem, and the third follows from (2) and the projection formula combined with (1). This finishes the proof.
\end{proof}

Finally, we prove the main theorem.

\begin{proof}[Proof of Theorem \ref{maintheorem}] We can and do assume the additional assumptions (3) and (4) from Section \ref{statement} without loss of generality. In particular, we assume that $K$ is neat and that there exists some connected component $X^+$ of $X$ such that $Z$ is contained in the connected component $S_K(\mathbf G,X^+)$ of $\Sh_K(\mathbf G,X)$. 
	
	We fix a faithful representation $\rho:\mathbf G\rightarrow\GL_n$ and we let $c$ and $\delta$ denote the effectively computable, positive constants afforded to us by Theorem \ref{mainbound}. We let $k$ and $f$ denote the effectively computable, positive integers afforded to us by Theorem \ref{hecke} and we let $d$ denote the effectively computable constant afforded to us by Lemma \ref{lowerbound}.

Let $V$ denote a non-facteur maximal special subvariety of $Z$.

Furthermore, let $(\mathbf P,X_\mathbf P)$ be a Shimura subdatum of $(\mathbf G,X)$ and let $X^+_\mathbf P$ be a connected component of $X_\mathbf P$ contained in $X^+$ such that $Z$ is a Hodge generic subvariety of the image of $S_{K(\mathbf P)}(\mathbf P,X_\mathbf P^+)$ in $S_{K}(\mathbf G,X^+)$ under the induced morphism
\begin{align*}
\iota_\mathbf P:\Sh_{K(\mathbf P)}(\mathbf P,X_\mathbf P)\rightarrow\Sh_K(\mathbf G,X).
\end{align*}

Let $V_\mathbf P$ denote an irreducible component of $\iota_\mathbf P^{-1}(V)$ contained in $S_{K(\mathbf P)}(\mathbf P,X_\mathbf P^+)$ and let $Z_\mathbf P$ be an irreducible component of $\iota_\mathbf P^{-1}(Z)$ containing $V_\mathbf P$. We can and do assume that  $Z_P$ is a proper subvariety of $S_{K(\mathbf P)}(\mathbf P,X_\mathbf P^+)$ since, otherwise, there is nothing to prove (by the maximality of $V$). As $V_\mathbf P$ is a non-facteur maximal special subvariety of $Z_\mathbf P$, Proposition \ref{induction} implies that, for any prime $p\notin\Sigma_{V}$ such that $K_p=\mathbf G_{\ZZ_p}(\ZZ_p)$, there exists an irreducible subvariety $Y_\mathbf P\subsetneq Z_\mathbf P$ containing $V_\mathbf P$ such that
\begin{align*}
\deg_{L_{K(\mathbf P)}}Y_\mathbf P\leq p^{(2f+k)}(\deg_{L_K(\mathbf P)}Z_\mathbf P)^2.
\end{align*}
By Lemma \ref{upperbound}, we have
\begin{align*}
\deg_{L_{K(\mathbf P)}}Z_\mathbf P\leq\deg_{L_{K}}Z.
\end{align*}
Setting $Y:=\iota_\mathbf P(Y_\mathbf P)\subsetneq Z$, we obtain a subvariety containing $V$, which satisfies
\begin{align*}
\deg_{L_K}Y\leq d\cdot \deg_{L_{K(\mathbf P)}}Y \leq dp^{2f+k}(\deg_{L_K(\mathbf{P})}Z)^2\leq dp^{(2f+k)}(\deg_{L_{K}}Z)^2
\end{align*}
by Lemma \ref{lowerbound}, used in the first inequality, and Lemma \ref{upperbound}, used in the third inequality.
Iterating this procedure at most $\dim Z-2$ times, we deduce that
\begin{align}\label{ineq}
\deg_{L_{K}}V\leq [dp^{(2f+k)}]^{(\dim Z-2)}(\deg_{L_K}Z)^{2^{\dim Z-2}}.
\end{align}

On the other hand, by Theorem \ref{mainbound},
\begin{align}\label{ineq2}
\deg_{L_{K}}V\geq c_1\Pi_V^{\delta}.
\end{align}
Let $N:=N(K)$ be such that, for all primes $p\geq N$, we have $K_p=\mathbf G_{\ZZ_p}(\ZZ_p)$. Recall that we are allowed to choose any $p$ not dividing $\Pi_V$ such that $p\geq N$. To that end, fix an $\epsilon\in(0,\frac{1}{2})$. Then, by the Prime Number Theorem, there exists an absolute, effectively computable, positive constant $c(\epsilon)\in(0,1)$ such that, for any $x\geq 2$, there are more than $c(\epsilon)x^{1-\epsilon}$ primes less than $x$. On the other hand, the number of primes dividing $\Pi_V$ is at most $(\log 2)^{-1}\log\Pi_V$. Therefore, it follows from (\ref{ineq2}) that there exists a prime
\begin{align*}
p\in[N,c(\epsilon)^{-1/(1-\epsilon)}[(\log (c^{-1}_1\cdot\deg_{L_{K}}V)/\delta\log 2)+N+1]^{1/(1-\epsilon)}]
\end{align*}
not dividing $\Pi_V$. Plugging such a prime $p$ into (\ref{ineq}) concludes the proof.
\end{proof}

\bibliography{DegreeBound}
\bibliographystyle{alpha}

\end{document}